\documentclass[letterpaper,11pt]{amsart}
\usepackage{tikz}
\usetikzlibrary{arrows}
\usepackage{subcaption} 
\usepackage[font=footnotesize,labelfont=bf]{caption}
\usepackage{blkarray}
\usepackage{enumitem}
\usepackage{latexsym,array,delarray,epsfig,setspace,cleveref, mathtools,amssymb}
\theoremstyle{plain}
\newtheorem{thm}{Theorem}[section]
\newtheorem{lemma}[thm]{Lemma}
\newtheorem{prop}[thm]{Proposition}
\newtheorem{cor}[thm]{Corollary}

\newtheorem{problem}[thm]{Problem}

\newtheorem*{thm*}{Theorem}
\newtheorem*{lemma*}{Lemma}
\newtheorem*{prop*}{Proposition}
\newtheorem*{cor*}{Corollary}
\newtheorem*{conj*}{Conjecture}
\newtheorem*{alg*}{Algorithm}

\theoremstyle{definition}
\newtheorem{defn}[thm]{Definition}
\newtheorem{ex}[thm]{Example}

\newtheorem{rmk}[thm]{Remark}
\newtheorem{obs}[thm]{Observation}

\theoremstyle{remark}

\newcommand{\nn}{\mathbb{N}}

\newcommand{\rr}{\mathbb{R}}

\newcommand{\bfa}{\mathbf{a}}
\newcommand{\bfb}{\mathbf{b}}
\newcommand{\bfc}{\mathbf{c}}
\newcommand{\bfd}{\mathbf{d}}
\newcommand{\bfe}{\mathbf{e}}

\newcommand{\bfi}{\mathbf{i}}
\newcommand{\bfj}{\mathbf{j}}

\newcommand{\bfp}{\mathbf{p}}
\newcommand{\bfq}{\mathbf{q}}
\newcommand{\bfr}{\mathbf{r}}
\newcommand{\bft}{\mathbf{t}}
\newcommand{\bfu}{\mathbf{u}}
\newcommand{\bfv}{\mathbf{v}}
\newcommand{\bfw}{\mathbf{w}}
\newcommand{\bfx}{\mathbf{x}}
\newcommand{\bfy}{\mathbf{y}}
\newcommand{\bfz}{\mathbf{z}}

\newcommand{\bv}{\mathbf{v}}
\newcommand{\bd}{\mathbf{d}}
\newcommand{\bc}{\mathbf{c}}
\newcommand{\bb}{\mathbf{b}}
\newcommand{\ba}{\mathbf{a}}
\newcommand{\bu}{\mathbf{u}}
\newcommand{\bw}{\mathbf{w}}
\newcommand{\x}{\mathbf{x}}
\newcommand{\bx}{\mathbf{x}}
\newcommand{\bq}{\mathbf{q}}
\newcommand{\by}{\mathbf{y}}
\newcommand{\bdel}{\mathbf{\delta}}
\newcommand*\Bell{\ensuremath{\boldsymbol\ell}}
\newcommand*\bal{\ensuremath{\boldsymbol\alpha}}

\newcommand{\cH}{\mathcal{H}}
\newcommand{\cM}{\mathcal{M}}
\renewcommand{\O}{\mathcal{O}}
\newcommand{\U}{\mathcal{U}}
\newcommand{\cC}{\mathcal{C}}

\newcommand{\R}{\mathbb{R}}

\newcommand{\K}{\mathbb{K}}

\newcommand{\dig}{G^{m,n}_2}
\newcommand{\Dmn}{D_{m,n}}
\newcommand{\Ev}{\mathrm{Ev}}

\renewcommand{\phi}{\varphi}
\newcommand{\D}{\Delta}
\newcommand{\supp}{\mathrm{supp}}
\newcommand{\Corr}{\mathrm{Corr}}
\newcommand{\Cut}{\mathrm{Cut}}
\newcommand{\Gcut}{\mathrm{GCut}}
\newcommand{\Marg}{\mathrm{Marg}}

\newcommand{\conv}{\mathrm{conv}}
\newcommand{\switch}{\mathrm{switch}}
\renewcommand{\deg}{\mathrm{deg}}

\newcommand{\facet}{\mathrm{facets}}
\newcommand{\overbar}[1]{\mkern 1.5mu\overline{\mkern-1.5mu#1\mkern-1.5mu}\mkern 1.5mu}
\renewcommand{\bar}{\overbar}

\title{Generalized Cut Polytopes for Binary Hierarchical Models}
\author{Jane Ivy Coons, Joseph Cummings, Benjamin Hollering, and Aida Maraj}

\begin{document}

\maketitle

\begin{abstract}
  Marginal polytopes are important geometric objects that arise in statistics as the polytopes underlying hierarchical log-linear models. These polytopes can be used to answer geometric questions about these models, such as determining the existence of maximum likelihood estimates or the normality of the associated semigroup.  Cut polytopes of graphs have been useful in analyzing binary marginal polytopes in the case where the simplicial complex underlying the hierarchical model is a graph. We introduce a generalized cut polytope that is isomorphic to the binary marginal polytope of an arbitrary simplicial complex via a generalized covariance map.
  This polytope is full dimensional in its ambient space and has a natural switching operation among its facets that can be used to deduce symmetries between the facets of the correlation and binary marginal polytopes.
  We find complete $\cH$-representations of the generalized cut polytope for some important families of simplicial complexes. We also compute the volume of these polytopes in some instances.
\end{abstract}

\section{Introduction}
Hierarchical models for discrete random variables are one of the most commonly used tools for analyzing categorical data. They are defined by a simplicial complex $\Delta$ whose vertices represent random variables and a vector of states $r$ which encodes the number of states each random variable has. Hierarchical models are log-linear models and thus have an associated polytope, commonly called the \emph{marginal polytope}, which can be used to answer geometric questions concerning the model. For example, consider the  question of determining if there exists a table $u$ with given $\Delta$-marginals, $t$. Such a table $u$ exists if the vector of marginals $\frac{t}{N}$, where $N$ is the sample size, lies in the marginal polytope. This can be checked efficiently with an $\cH$-representation of the marginal polytope \cite{reducCyclic2002}. 

There is also the problem of determining if the maximum likelihood estimate (MLE) exists for a given table of counts $u$ and hierarchical model $(\Delta, r)$.  In many modern applications, hierarchical models are used to model data that takes the form of large sparse contingency tables. In such cases it is possible that maximum likelihood estimate (MLE) may fail to exist \cite{hierMLE2012}. The MLE for these models is guaranteed to exist if $\frac{t}{N}$ lies in the relative interior of the marginal polytope where $t$ is again the vector of $\Delta$-marginals. This can also be checked efficiently with an $\cH$-representation of the marginal polytope or a description of the facial sets \cite{hierMLE2012}. Recently, Wang, Rauh, and Massam gave an even more efficient algorithm that approximates the facial sets instead \cite{approxMargFaces2019}.

$\cH$-representations are also instrumental for determining if the marginal polytope or its associated cone are normal. Normality is an important property for a variety of reasons. Deciding membership in the semigroup $\nn A$ can be done in polynomial time when $A$ is normal but is NP-complete in general \cite{bernstein2015}. Normality of an associated semigroup is also important for the higher codimension toric fiber product described in  \cite{higherCodimTFP}. When $\Delta$ is a graph free of $K4$ minors, the marginal cone has an $\cH$-representation given by the cycles in the graph which was used in \cite{normalGraphOhsugi2010,normalGraphSullivant2009} to completely characterize when the marginal cone is normal. The $\cH$-representation in this case is a result of the $\cH$-representation of cut polytopes of these graphs. 

Obtaining a complete $\cH$-representation of the marginal polytope can be difficult for a number of reasons. The marginal polytope is typically not full-dimensional in its ambient space. Hence there is not a unique $\cH$-representation and as the dimension grows it becomes increasingly untenable to compute directly. 

It is well-known in the statistics literature that the binary marginal polytope of $\Delta$ is isomorphic to the correlation polytope (also known as the moment polytope) of $\D$ \cite[Proposition 19.1.20]{algstat2018}. The correlation polytope has coordinates indexed by the nonempty faces of $\Delta$. The points in the correlation polytope are vectors of $\Delta$-moments of probability distributions of $n$-dimensional binary random vectors \cite[Proposition~8.2.7]{algstat2018}. The $\cH$-representation of the correlation polytope is known for graphs free of $K_4$ minors \cite{Barahona1986}; however, there is no known facet description for arbitrary $\Delta$. In the case where $\Delta$ is a graph, the binary marginal polytope and correlation polytope of $\Delta$ and the cut polytope of the suspension of $\Delta$ are all isomorphic to one another.

In this paper we introduce a generalization of the cut polytope that extends to any simplicial complex and use the additional structure to obtain $\cH$-representations for some families of simplicial complexes. This new polytope is full dimensional in its ambient space; thus its facet description is unique. The switching operation of \cite[Chapter 26.3]{cutsDeza1997} can also be used to create new facets from known facets, which is a technique we exploit throughout this paper. In Section 2, we provide background on marginal, correlation, and cut polytopes. In Section 3, we introduce the generalized cut polytope for a simplicial complex and describe the switching operation on it. We use this operation to deduce symmetries among the facets of the correlation and marginal polytopes. In Sections 4, 5 and 6 we give $\cH$-representations of the generalized cut polytope for different families of simplicial complexes including the boundary of the simplex and some other unimodular simplicial complexes. Section 7 provides some results on the degree of some hierarchical models and hence the normalized volume of the generalized cut polytope for some families of simplicial complexes. 
 
\section{Preliminaries}
In this section we give some background on hierarchical models and the marginal polytope. We refer the reader to \cite[Chapter 9]{algstat2018} for additional details on hierarchical models and to \cite{ziegler} for more information on polyhedra and simplicial complexes. Throughout the present paper, we denote the entry in the $i$th row and $j$th column of a matrix $A$ by $A_i^j$.

\begin{defn}
Let $2^{[n]}$ denote the power set of $[n]$. A \emph{simplicial complex} on ground set $[n]$ is a set $\Delta \subset 2^{[n]}$ such that if $F \in \Delta$ and $F' \subset F$ then $F' \in \Delta$. The elements of $\Delta$ are called \emph{faces} and the inclusion maximal elements are called \emph{facets}. We denote the set of facets of $\Delta$ with $\facet(\Delta)$. We denote by $\bar\Delta$ the set of nonempty faces of $\Delta$; that is $\bar\Delta = \Delta \setminus \{ \emptyset \}.$
\end{defn}

We typically define simplicial complexes by their facets and will often drop the set brackets when writing sets and facets. For instance, we write $\Delta = [12][23]$ for the simplicial complex with facets $\{1, 2\}$ and $\{2, 3\}$ which we abbreviate by $12$ and $23$.

A \emph{hierarchical model} on random variables $X_1, \ldots X_n$ is a \emph{log-linear model} defined by a simplicial complex $\Delta$ on $[n] = \{1, 2, \ldots n\}$ and a vector of states $r = (r_1, \ldots, r_n) \in \nn^n$. Each random variable $X_i$ is naturally associated to the vertex $i$ of $\Delta$ and has $r_i$ states. 
The design matrix of the model, denoted $U_{\Delta,r}$, is constructed in the following way. Let $\bfi = (i_1, \ldots, i_n) \in \prod_{k = 1}^n[r_k]$ and for any $F \in \facet(\Delta)$ let $\bfi_F$ be the restriction of $\bfi$ to the indices in $F$. The columns of $U_{\Delta, r}$ are indexed by $\bfj \in \prod_{k = 1}^n[r_k]$ and the rows are indexed by pairs $(\bfi_F,F)$ for $\bfi_F \in \prod_{k \in F}[r_k]$. The entries are given by
\begin{equation}
\label{eqn:margEntries}
    u_{(\bfi_F,F)}^\bfj = 
\begin{cases}
  1 & \text{ if } \bfj_F = \bfi_F \\
  0 & \text{otherwise}.
\end{cases}
\end{equation}

For more background on log-linear models see \cite[Chapter 6]{algstat2018}.
In this paper we restrict to the case where each $r_i = 2$ so that the associated random variables are all binary and we use the notation
$U_\Delta = U_{\Delta, (2, 2, \ldots 2)}$. This is called the \emph{binary hierarchical model} associated to $\Delta$. In this case, the column indices $\bfj$ are 0/1 strings of length $n$ which are naturally in bijection with subsets of $[n]$. We reformulate the definition of the matrix $U_\Delta$ which simplifies notation. With this convention the columns of $U_\Delta$ are indexed by $S \subset [n]$ and denoted $\bu^S$ while the rows are indexed by pairs $(H,F)$ with $F \in \facet(\Delta)$ and $H \subset F$. Then the entries are given by
\[
u_{(H, F)}^S =
\begin{cases}
  1 & \text{ if } S \cap F = H \\
  0 & \text{otherwise}. 
\end{cases}
\]
We let $\mathcal{E}(\Delta) : = \{(H,F) \mid H \subset F \in \facet(\Delta) \}$ be the set of row indices of $U_{\Delta}$.

\begin{defn}
The \emph{marginal polytope} is the convex hull of the columns of the matrix $U_{\Delta,r}$. We denote this with $\Marg(\Delta,r)$. In the case where $r = (2,\dots,2)$, we call this the \emph{binary marginal polytope} and denote it
$\Marg(\Delta) = \conv(U_{\Delta}) = 
\{\sum_{S} \lambda_{S} \bu^{S} ~|~ \sum_{S} \lambda_{S} = 1, ~\lambda_{S} \geq 0 \}$.
\end{defn}

The following example illustrates how this design matrix is constructed.

\begin{ex}
\label{ex:margMatrix}
Let $\Delta = [12] [23]$ be a simplicial complex. Then the matrix $U_\Delta$ is
\[
U_\Delta = 
\begin{blockarray}{ccccccccc}
      & \emptyset & 1 & 2 & 3 & 12 & 13 & 23 & 123 \\
      \begin{block}{c(cccccccc)}
         (\emptyset, 12)   & 1 & 0 & 0 & 1 & 0 & 0 & 0 & 0 \\
        (1,12)           & 0 & 1 & 0 & 0 & 0 & 1 & 0 & 0 \\
        (2,12)          & 0 & 0 & 1 & 0 & 0 & 0 & 1 & 0 \\
        (12,12)          & 0 & 0 & 0 & 0 & 1 & 0 & 0 & 1 \\
        (\emptyset,23)   & 1 & 1 & 0 & 0 & 0 & 0 & 0 & 0 \\
        (2,23)           & 0 & 0 & 1 & 0 & 1 & 0 & 0 & 0 \\
        (3,23)           & 0 & 0 & 0 & 1 & 0 & 1 & 0 & 0 \\
        (23,23)          & 0 & 0 & 0 & 0 & 0 & 0 & 1 & 1 \\
      \end{block}
\end{blockarray}.
\]
The first block of rows corresponds to the facet $12$ while the second block corresponds to the facet $23$. Within each block, the rows are indexed by subsets of the corresponding facet.
\end{ex}

This presentation of the polytope as a convex hull of a finite set is called a $\mathcal{V}$-representation of the polytope. Every polytope also can be defined as an intersection of finitely many half-spaces and such a presentation is called an $\cH$-representation \cite{ziegler}. An $\cH$-representation can be constructed from the $\mathcal{V}$-representation using Fourier-Motzkin elimination but this becomes intractable as the size of the polytope grows. As explained in the introduction, $\cH$-representations of the marginal polytope are desirable since they allow one to easily answer many relevant statistical questions about the corresponding hierarchical model. This motivates the following problem for marginal polytopes. 

\begin{problem}
\label{pr:margHrep}
Given a $\mathcal{V}$-representation, $\Marg(\Delta) = \conv(U_\Delta)$ efficiently construct an 
\newline
$\cH$-representation of $\Marg(\Delta)$; that is find a matrix $A$ and vector $\bfa$ such that
$\Marg(\Delta) = \{\bfz ~|~ A\bfz \leq  \bfa\}$.
\end{problem}

We may also call an $\cH$-representation of the polytope a facet description of the polytope. The example below illustrates such an $\cH$-representation. 

\begin{ex}
Let $\Delta = [12] [23]$. Then $U_\Delta$ is the same matrix given in Example \ref{ex:margMatrix} so $\Marg(\Delta) = \conv(U_\Delta)$. The coordinates of the ambient space of the polytope are indexed by pairs $\mathcal{E}(\Delta)$.  We denote these coordinates with $z_{(H,F)}$. An $\cH$-representation of $\Marg(\Delta)$ is the set of $\bfz \in \rr^8$ such that:
\begin{gather*}
    z_{(\emptyset,23)} + z_{(3,23)} + z_{(2,23)} + z_{(23,23)} = 1, \\
    z_{(\emptyset,12)} + z_{(1,12)} + z_{(2,23)} + z_{(23,23)} = 1, \\
    z_{(2,12)} + z_{(12,12)} - z_{(2,23)} - z_{(23,23)} = 0, \\
    z_{(3,23)} \geq 0, ~~ 
    z_{(2,23)} \geq 0, ~~
    z_{(23,23)} \geq 0, ~~
    z_{(2,12)} \geq 0, ~~
    z_{(\emptyset,12)} \geq 0, \\
    z_{(2,23)} + z_{(23,23)} - z_{(2,12)} \geq 0, ~~
    z_{(3,23)} + z_{(1,12)} + z_{(23,23)} \leq 1, ~~
    z_{(\emptyset,12)} + z_{(1,12)} + z_{(23,23)} \leq 1.
\end{gather*}
In this case, $\dim(\Marg(\Delta)) = 5$ but it sits in an 8-dimensional ambient space. This means that the above $\cH$-representation is not unique. 
\end{ex}

Problem \ref{pr:margHrep} has been solved for some families of simplicial complexes. For instance, \cite{Barahona1986} describes the facet-defining inequalities for some graphs using the cycles in the graph. An $\cH$-representation can also be constructed inductively when $\Delta$ is decomposable \cite{compPoly06}.  When all of the random variables in the model are binary, the marginal polytope is affinely isomorphic to the \emph{correlation polytope} which is full dimensional in its ambient space \cite[Proposition 19.1.20]{algstat2018}. We explicitly describe this isomorphism in the appendix. Recall that $\bar\D$ is the set of nonempty faces of $\D$.

\begin{defn}
\label{defn:corrPoly}
The \emph{correlation polytope} associated to a simplicial complex $\Delta$, denoted $\Corr(\Delta)$ is a $0/1$-polytope in $\R^{\bar\Delta}$ whose vertices $\bv^S$ are indexed by subsets $S$ of $[n]$. 
For $F\in \bar\D$, the $F$ coordinate of $\bv^S$ is
\[
v^S_F = \begin{cases}
1 & \text{ if } F \subset S \\
0 & \text{ otherwise.}
\end{cases}
\]
The matrix with columns $\bv^S$ is denoted $V_{\Delta}$.
\end{defn}

The affine isomorphism between $\Marg(\Delta)$ and $\Corr(\D)$ maps the vertex $\bu^S$ of $\Marg(\Delta)$ to the vertex $\bv^S$ of the correlation polytope, see \Cref{Prop:MargToCorr}. 

\begin{ex}
Again let $\Delta = [12][23]$. $\Corr(\D)$ has 8 vertices which are indexed by subsets $S$ of $[3]$. It sits in an ambient space of dimension $5$ with coordinates indexed by the 5 nonempty faces of $\Delta$ which are $1, 2, 3, 12, 23$. These vertices are the columns of the matrix
\[
V_\Delta = 
\begin{blockarray}{ccccccccc}
        & \emptyset & 1 & 2 & 3 & 12 & 13 & 23 & 123\\
      \begin{block}{c(cccccccc)}
        1   & 0 & 1 & 0 & 0 & 1 & 1 & 0 & 1 \\
        2   & 0 & 0 & 1 & 0 & 1 & 0 & 1 & 1 \\
        3   & 0 & 0 & 0 & 1 & 0 & 1 & 1 & 1 \\
        12  & 0 & 0 & 0 & 0 & 1 & 0 & 0 & 1 \\
        23  & 0 & 0 & 0 & 0 & 0 & 0 & 1 & 1 \\
      \end{block}
    \end{blockarray}.
\]
It is full dimensional in its ambient space and so it has unique $\cH$-representation as the set of $\by \in \rr^{\bar\Delta}$ such that 
\begin{align*}
y_{2} + y_{3} - y_{23} \leq 1, &\quad&
y_{1} + y_{2} - y_{12} \leq 1, &\quad&
y_{12} - y_{1} \leq 0, &\quad&
y_{12} - y_{2} \leq 0, \\
y_{23} - y_{2} \leq 0, &\quad&
y_{23} - y_{3} \leq 0, &\quad&
y_{12} \geq 0, &\quad&
y_{23} \geq 0.
\end{align*}
\end{ex}


When the simplicial complex $\Delta$ is a graph there is another isomorphic polytope, named the \emph{cut polytope}. 

\begin{defn}
Let $\Gamma$ be a graph with vertex set $[n]$ and edge set $E$. The \emph{cut polytope} has vertices $\bdel^{S | T} \in \rr^E$ indexed by set partitions $S | T$ of $[n]$ given by
\[\bdel^{S | T}_e=\begin{cases} 
      1 & \#(e \cap S)=1 \\
      0 & \mathrm{else}.
   \end{cases}
\]
for each edge $e \in E$.
Note the definition does not depend on which part of the set partition $S|T$ is used.
\end{defn}

The \emph{covariance map} $\varphi_\Gamma$ gives an isomorphism from $\Corr(\Gamma)$ to $\Cut(\Hat{\Gamma})$ where $\Hat{\Gamma}$ is the suspension of the graph $\Gamma$ \cite[Chapter 5.2]{cutsDeza1997}.  
The \emph{suspension} $\Hat{\Gamma}$ is a graph obtained from $\Gamma$ by adding a new vertex, $n+1$, to $\Gamma$ and connecting every existing vertex to the new one. The covariance map is defined by $\phi_{\Gamma}(\bfy)_{i,n+1} = y_{i}$ for $i \in [n]$ and $\phi_\Gamma(\bfy)_{i,j} = y_{i} + y_{j}-2y_{ij}$ for $ij \in \Gamma$.

The cut polytope has been studied extensively and appears in many different areas of math \cite{cutsDeza1997}. For instance, the elements of the cut cone of the complete graph are the $\ell_1$-embeddable finite metrics. There are also applications to optimization problems such as the max-cut problem. In \cite{Barahona1983}, Barahona shows that the max-cut problem is solvable in polynomial time for graphs with no $K_5$ minor.  

The cut polytope also has additional structure that can make it easier to find a facet description, though an explicit facet description for general graphs is unlikely \cite{cutsDeza1997}. The \emph{switching operation} allows one to find new facet-defining inequalities from a given one. This can make it easier to enumerate all of the inequalities that define the polytope. We describe this operation in Section \ref{sec:switching}. There is also a complete facet description for the polytope when the graph $\Delta$ is free of $K_5$ minors \cite{Barahona1986}. This has been further leveraged to prove results about normality of the marginal polytope and marginal cone \cite{normalGraphOhsugi2010, normalGraphSullivant2009}. Leveraging the additional structure of cut polytopes to obtain a better understanding of marginal polytopes is our main goal in this paper. In the next section we generalize the definition of the cut polytope of a graph to any simplicial complex.

\begin{figure}
    \centering
    \begin{tikzpicture}[scale = .75]
        \draw (0,0)--(1,0)--(2,0);
        \draw [fill] (0,0) circle [radius = .05];
        \draw [fill] (1,0) circle [radius = .05];
        \draw [fill] (2,0) circle [radius = .05];
        \node [below] at (0,0) {1};
        \node [below] at (1,0) {2};
        \node [below] at (2,0) {3};
        \node [below] at (1,-.5) {$\Delta$};
        
        \draw (4,0)--(5,0)--(6,0);
        \draw [fill] (4,0) circle [radius = .05];
        \draw [fill] (5,0) circle [radius = .05];
        \draw [fill] (6,0) circle [radius = .05];
        \draw [fill] (5,1) circle [radius = .05];
        \node [below] at (4,0) {1};
        \node [below] at (5,0) {2};
        \node [below] at (6,0) {3};
        \draw (4,0)--(5,1);
        \draw (5,0)--(5,1);
        \draw (6,0)--(5,1);
        \node [above] at (5,1) {4};
        \node [below] at (5, -.5) {$\Hat{\Delta}$};
    \end{tikzpicture}
    \caption{The graph $\Delta$ and its suspension $\Hat{\Delta}$ from Example \ref{ex:graphCutPoly}.}
    \label{fig:pathCutPoly}
\end{figure}
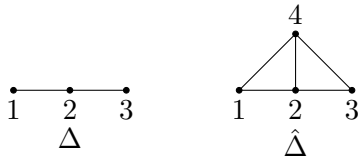

\begin{ex}
\label{ex:graphCutPoly}
Let $\Gamma$ be the graph on $[3]$ with edges $12$ and $23$ pictured in Figure \ref{fig:pathCutPoly}. Then the cut polytope has vertices indexed by set partitions $S | T$ of $[3]$ and coordinates indexed by the edges of $\Gamma$. The polytope is the convex hull of the columns of the matrix
\[
\begin{blockarray}{ccccc}
        & \emptyset | 123 & 1 | 23 & 12| 3 & 13 | 2 \\
      \begin{block}{c(cccc)}
        12   & 0 & 1 & 0 & 1  \\
        23   & 0 & 0 & 1 & 1  \\
      \end{block}
    \end{blockarray}.
\]
The cut polytope of the suspension $\Hat{\Gamma}$ has vertices indexed by set partitions $S | T$ of $[4]$, but note that these are in bijection with subsets $S \subset [3]$ since we can always choose $4$ to appear in $T$. Then $\Cut(\Hat{\Gamma})$ is the convex hull of the columns of the matrix
\[
\begin{blockarray}{ccccccccc}
        & \emptyset & 1 & 2 & 3 & 12 & 13 & 23 & 123\\
      \begin{block}{c(cccccccc)}
        14   & 0 & 1 & 0 & 0 & 1 & 1 & 0 & 1 \\
        24  & 0 & 0 & 1 & 0 & 1 & 0 & 1 & 1 \\
        34   & 0 & 0 & 0 & 1 & 0 & 1 & 1 & 1 \\
        12   & 0 & 1 & 1 & 0 & 0 & 1 & 1 & 0 \\
        23   & 0 & 0 & 1 & 1 & 1 & 1 & 0 & 0 \\
      \end{block}
    \end{blockarray}.
\]
The polytope $\Cut(\Hat{\Gamma})$ is full dimensional in its ambient space.
\end{ex}

As the last of our preliminaries, we give a brief introduction to Gale duality since it is used throughout the paper in order to understand which sets of vertices lie on a common face. For a more thorough survey of the subject, we recommend \cite{rekhathomas, ziegler}.

\begin{defn}\label{defn:Gale-Transform}
Let \(P \subset \R^{d-1}\) be a full-dimensional polytope with vertices \(\{\bfv_1, \dotsc, \bfv_n\}\). Let $V$ the be matrix with columns $\bfv_1, \dots, \bfv_n$. Let $\bar V$ denote the matrix obtained by adding a row of all ones to the top of $V$. Let \(B\) be a matrix whose columns \(\{B^1, \dotsc, B^{n-d}\}\) form a basis for the kernel of $\bar V$. The {\it Gale transform} of \(P\) is then the \(n\) ordered rows \( \{\bb_1,\dotsc, \bb_n\}\) of \(B\).
\end{defn}

The Gale transform is not unique since there is typically no unique choice of basis for \(\ker(V)\); however, the underlying oriented matroid is unique so we may choose any basis. The following theorem allows us to use the Gale transform of $P$ to determine which vertices lie on a common face of $P$.

\begin{thm}\label{thm:Gale-Transform}
Let \(P\subset \R^{d-1}\) be a full-dimensional polytope with vertices \(\{\bfv_1,\dotsc, \bfv_n\}\). A collection of vertices \(\{\bfv_j ~|~ j \in \mathcal{J}\}\) are the vertices of a face of \(P\) if and only if \({\bf 0} \in \mathrm{relint}(\conv\{\bb_i ~|~ i \notin \mathcal{J}\})\) or \(\mathcal{J} = [n]\).
\end{thm}

Let $\mathcal{J} \subset [n]$ such that $\conv \{\bv_i \mid i \in \mathcal{J} \}$ is a face (resp. facet) of $P$. Then we call $\{ \bv_j \mid j \not\in \mathcal{J}\}$ a \emph{co-face} (resp. \emph{co-facet}) of $P$.

\section{The Generalized Cut Polytope}
This section introduces  generalized cut polytopes for simplicial complexes. We show that they are isomorphic to the correlation polytopes and hence binary marginal polytopes. We also show that the switching operation in \cite[Chapter 26.3]{cutsDeza1997} extends to the generalized cut polytope. 

\subsection{Motivation and Construction} 
In the previous section we saw that when the simplicial complex $\Delta$ is a graph, we have
$\Marg(\Delta) \cong \Corr(\Delta) \cong \Cut(\Hat{\Delta})$. The last isomorphism was given by the covariance map, so a natural candidate for a generalization of the cut polytope is the image of the correlation polytope under a generalized covariance map. We begin with this definition. 



\begin{defn}
\label{def:gencovmap}
Given $\D$ a simplicial complex with ground set $[n]$, the  \emph{generalized covariance map} $\phi_{\D}$ is the linear transformation from $\R^{\overline{\Delta}}$ to itself defined by
\[
[\phi_{\Delta}(\bfy)]_F = \sum_{\substack{H: \\  \emptyset\neq H \subset F}} (-2)^{\#H -1} y_H.
\]
\end{defn}

\begin{ex}\label{Ex:RunningCov}
Consider the simplicial complex $\Delta = [123][234]$ pictured in Figure \ref{fig:turtle}. The generalized covariance map is the linear transformation defined by the matrix,
\[
\Phi_{\Delta} = \begin{blockarray}{cccccccccccc}
& 1 & 2 & 3 & 4 & 12 & 13 & 23 & 24 & 34 & 123 & 234 \\
\begin{block}{c(ccccccccccc)}
1&1&0&0&0&0&0&0&0&0&0&0\\
2&0&1&0&0&0&0&0&0&0&0&0\\
3&0&0&1&0&0&0&0&0&0&0&0\\
4&0&0&0&1&0&0&0&0&0&0&0\\
12&1&1&0&0&-2&0&0&0&0&0&0\\
13&1&0&1&0&0&-2&0&0&0&0&0\\
23&0&1&1&0&0&0&-2&0&0&0&0\\
24&0&1&0&1&0&0&0&-2&0&0&0\\
34&0&0&1&1&0&0&0&0&-2&0&0\\
123&1&1&1&0&-2&-2&-2&0&0&4&0\\
234&0&1&1&1&0&0&-2&-2&-2&0&4\\
\end{block}
\end{blockarray}.
\]
\end{ex}

In \Cref{thm:dim} we will show that the image of $\Corr(\Delta)$ under the generalized covariance map is the following polytope. 

\begin{defn}
The \emph{generalized cut polytope} of $\D$, denoted $\Gcut(\D)$, is a $0/1$-polytope in $\R^{\overline{\D}}$ defined as the convex hull of vertices $\{\bfd^{S}~|~ S\subset [n]\}$, where for every $F \in \bar\D$, the $F$-coordinate of $\bfd^S$ is defined by
\[
d^S_F = \begin{cases}
1 & \text{ if } \#(F\cap S) \text{ is odd} \\
0 &  \text{ if } \#(F\cap S) \text{ is even}.
\end{cases}
\]
The matrix whose columns are the vertices $\bd^S$ is denoted $D_{\D}$.
\end{defn}

\begin{ex}\label{Ex:RunningCut}
Again consider the simplicial complex $\Delta = [123][234]$. Then the generalized cut polytope $\Gcut(\Delta)$ is the convex hull of the columns of the matrix,
\[
D_{\Delta} = \begin{blockarray}{ccccccccccccccccc}
&\emptyset&1&2&3&4&12&13&14&23&24&34&123&124&134&234&1234 \\
\begin{block}{c(cccccccccccccccc)}
1&0&1&0&0&0&1&1&1&0&0&0&1&1&1&0&1\\
2&0&0&1&0&0&1&0&0&1&1&0&1&1&0&1&1\\
3&0&0&0&1&0&0&1&0&1&0&1&1&0&1&1&1\\
4&0&0&0&0&1&0&0&1&0&1&1&0&1&1&1&1\\
12&0&1&1&0&0&0&1&1&1&1&0&0&0&1&1&0\\
13&0&1&0&1&0&1&0&1&1&0&1&0&1&0&1&0\\
23&0&0&1&1&0&1&1&0&0&1&1&0&1&1&0&0\\
24&0&0&1&0&1&1&0&1&1&0&1&1&0&1&0&0\\
34&0&0&0&1&1&0&1&1&1&1&0&1&1&0&0&0\\
123&0&1&1&1&0&0&0&1&0&1&1&1&0&0&0&1\\
234&0&0&1&1&1&1&1&1&0&0&0&0&0&0&1&1\\
\end{block}
\end{blockarray}.
\]
Let $V_{\Delta}$ be the matrix of the correlation polytope of $\Delta$ and let $\Phi_{\Delta}$ be the generalized covariance map as in  \Cref{Ex:RunningCov}. Then we note that $\Phi_{\Delta} V_{\Delta} = D_{\Delta}.$
\end{ex}

Now we are ready to state the main result of this section. 
\begin{thm}
\label{thm:dim}
The generalized cut polytope, the correlation polytope, and the binary marginal polytope for a simplicial complex $\D$ are isomorphic to each other. Their dimension equals the number of nonempty faces of $\Delta$.
\end{thm}

\begin{proof}

Fix a simplical complex $\D$. The matrix $\Phi_{\D}\in \R^{\bar\D \times \bar\D}$ for the generalized covariance map has entries
\[
(\Phi_{\D})^H_F= \begin{cases}{}
(-2)^{\#H -1} & if \: H\subset F\\
0 & if \: H\not\subset F.
\end{cases}
\]
First, the transformation $\phi_{\D}$ is an isomorphism as the matrix $\Phi_{\D}$ is a lower triangular matrix with non-zero entries on the diagonal. It remains to show that the image of  the correlation polytope of $\D$ under this map is the generalized polytope of $\D$. 

Fix $S\subset [n]$. We will prove that $\phi_{\D}$ sends the vertex $\bfv^S$ of $\Corr(\D)$ to the vertex $\bfd^S$ of $\Gcut(\D)$. Take $F\in\bar\D$. If $\#F=1$ then $\phi_{\D}(\bfv^S)_F=(-2)^0 v^S_F$ since $F$ is the only non-empty subset of itself. Hence, 
\[
		\phi(\bfv^S)_F = \begin{cases}
			1 & \text{ if } F \subset S \\
			0 & \mathrm{else}
		\end{cases} \quad = \quad \begin{cases}
 			1 & \text{ if } \#(F\cap S) \text{ is odd} \\
			0 &  \text{ if } \#(F\cap S) \text{ is even} 
		\end{cases}
		 \quad = \quad d^S_F.
\]

We proceed by induction. Assume that $\phi_{\D}(\bfv^S)_F=d^S_F$ for any $F$ of size less than $k,~1\leq k< n$. Let $F$ be a  non-empty set of size $k$ in $\D$. Without loss of generality let $F=[k]$. \\
\textit{Case 1}: Suppose that $F \not\subset S$. Without loss of generality, assume that  $k$ is not in $ S$. Since $v^S_{H} = 0$ for any non-empty subset $H$ of $F$ that contains $k$, we have the following:
	\begin{align*}
		\phi_{\D}(\bfv^S)_{F}
		&= \sum_{\substack{\emptyset\neq H\subset F,\\ k\notin H}}(-2)^{\# H-1}v^S_{H}
		 &=\sum_{\emptyset\neq H\subset [k-1]}(-2)^{\# H-1}v^S_{H}
		&\quad =  	\phi_{\Delta}(\bfv^S)_{[k-1]}.
		\end{align*}
		
The induction hypothesis on the size of $F$ and that $F\cap S=[k-1]\cap S$ induces the following:			
		\[
		\phi_{\Delta}(\bfv^S)_{[k-1]} \quad = \quad \begin{cases}
			1 & \text{ if } \#([k-1]\cap S) \text{ is odd} \\
			0 &  \text{ if } \#([k-1]\cap S) \text{ is even} 
		\end{cases} \\ \quad  = \quad \begin{cases}
			1 & \text{ if } \#(F\cap S) \text{ is odd} \\
			0 &  \text{ if } \#(F\cap S) \text{ is even},
		\end{cases}
	\]
which ends this case.\\ 			
\textit{Case 2}: If $F \subset S$, then $\#(F \cap S) = k$, and for all $H \subset F$ we have that  $v^S_{H} = 1$. Hence,
	\begin{align*}
		\phi_{\D}(\bfv^S)_F &= \sum_{\substack{H \subset F \\ H\neq \emptyset}} (-2)^{\#H - 1}
		 \quad = \quad  \sum_{i=1}^k \binom{k}{i} (-2)^{i-1}
		\quad = \quad \frac{(1+(-2))^k-1}{-2}
&= \begin{cases}
			1 & \text{ if } k \text{ is odd} \\
			0 &  \text{ if } k \text{ is even}, 
		\end{cases}
\end{align*}
as needed.

  The polytopes $\Corr(\D)$ and $\Marg(\D)$ are isomorphic  by \cite[Proposition~19.1.20]{algstat2018}.  The dimension of these polytopes being  $\#\D-1$ is a consequence of the fact that the correlation polytope is full-dimensional. 
\end{proof}

Note that for a graph $\Gamma$ on $n$ vertices, we do \emph{not} have that $\mathrm{\Gcut}(\Gamma)\cong \mathrm{Cut}(\Gamma)$. 
However, we do have that $\Gcut(\Gamma) \cong \mathrm{Cut}(\hat{\Gamma})$. We can see this by simply  relabeling the $ \{i,n+1\}$ coordinates of $\mathrm{Cut}(\hat{\Gamma})$ with $\{i\}$. One can check this using the simplicial complex in \Cref{ex:graphCutPoly}.

\subsection{The Switching Operation}
\label{sec:switching}
Here we describe the switching operation on the facets of the generalized cut polytope. As described in \cite[Section~26.3]{cutsDeza1997}, this is a consequence of the fact that the symmetric difference of two cuts is again a cut, and that this property applies more generally to the set families that are closed under taking symmetric differences. We first present the switching on the facets of the generalized cut polytope, and then prove that this operator yields new facet-defining inequalities.

\begin{defn}
Let $\bfa \cdot \bfx\leq a_0$ be a valid inequality for $\Gcut(\D)$, and let $I \subset [n]$. We define the map  $\bfa^{(I)}$ on $R^{\bar\D}$ by $a^{(I)}_{F} = (-1)^{\#(I \cap F)} a_F.$
The \emph{switching} of the inequality $\bfa \cdot \bfx \leq a_0$ with respect to the set $I$ is the inequality
\[
\bfa^{(I)} \cdot \bfx \leq a_0 - \bfa \cdot \bfd^I.
\]

More generally, given a set of subsets $\mathcal{I}$ in $\bar\D$ and a valid inequality $\bfa \cdot \bfx \leq a_0$ for  $\Gcut(\D)$ define  $\switch_{\mathcal{I}}(\bfa \cdot \bfx \leq a_0)=\{\bfa^{(I)} \cdot \bfx \leq a_0 - \bfa \bfd^I\:|\: I\in\mathcal{I}\}.$
\end{defn}

This notation diverges slightly from that given in \cite[Chapter~26]{cutsDeza1997}. Indeed, our definition of switching with respect to $I$ corresponds to a switching with respect to $\{T \mid \#(T \cap I) \text{ odd}\}$ in the language of \cite{cutsDeza1997}.

For each $S \subset [n]$, we denote by $\supp(\bfd^S)$ the support of $\bfd^S$. That is, $\supp(\bfd^S) = \{ F \in \D \mid \#(F \cap S) \text{ is odd}\}.$

\begin{prop}\label{Prop:SymmDiff}
The set of all supports of vertices of $\Gcut(\D)$ is closed under taking symmetric differences.
\end{prop}

To facilitate this proof and many of the proofs in the present paper, we introduce the following notation. Denote by $\mathbb{F}_2$ the finite field of order 2. For each $S \subset [n]$, denote by $\bfi^S$ the indicator vector of $S$ in the vector space $\mathbb{F}^n_2$.

\begin{proof}[Proof of \Cref{Prop:SymmDiff}]
Let $S, T \subset [n]$ such that $S \neq T$. Then we have $\bfi^{S \triangle T} = \bfi^S + \bfi^T$.  We will show that $\supp(\bd^S) \triangle \supp(\bd^T) = \supp(\bd^{S \triangle T})$. 

Note that $\bfi^S \cdot \bfi^A = 0$ if $\#(S \cap A)$ is even and $\bfi^S \cdot \bfi^A = 1$ if $\#(S \cap A)$ is odd, and similarly for $T$. We have $\bfi^{S \triangle T} \cdot \bfi^A = \bfi^S \cdot \bfi^A + \bfi^T \cdot \bfi^A$. So $\#((S \triangle T) \cap A)$ is odd if and only if $\#(S \cap A)$ and $\#(T \cap A)$ have opposite parities. Therefore $\supp(\bd^S) \triangle \supp(\bd^T) = \supp(\bd^{S \triangle T})$.
\end{proof}

\begin{cor}[\cite{cutsDeza1997}, Corollary 26.3.5]\label{Cor:SwitchingWorks}
Let $I \subset [n]$. The inequality $\ba \cdot \bx \leq a_0$ is valid (resp. facet-defining) for $\Gcut(\Delta)$ if and only its switching with respect to $I$, $\ba^{(I)} \cdot \bx \leq a_0 - \ba \cdot \bd^I$ is valid (resp. facet-defining) for $\Gcut(\Delta)$.
\end{cor}

\begin{rmk}\label{mod-switching}
By \cite[Proposition~26.3.6]{cutsDeza1997}, all of the facet-defining inequalities of $\Gcut(\Delta)$ can be obtained via switching from the inequalities defining facets that contain the origin. In the present paper, we often take the reverse perspective; for several simplicial complexes $\Delta$, we characterize those facet-defining hyperplanes that do \emph{not} contain the origin. Since switching an inequality with respect to a fixed set twice yields the original inequality, \cite[Proposition~26.3.6]{cutsDeza1997} implies that every facet-defining hyperplane of $\Gcut(\Delta)$ can also be obtained via switching from those hyperplanes that do not contain the origin. Indeed, given a homogeneous facet-defining inequality $\bfa \bfx \leq 0$ for $\bfa \neq \mathbf{0}$, we can always find an $I$ such that $\bfa \bfd^I \neq 0$ as $\bfa$ defines a proper face of $\Gcut(\Delta)$. Switching with respect to $I$ gives an inhomogeneous inequality.
\end{rmk}

Fix a simplicial complex $\Delta$ and let $\Phi_{\D}$ be the matrix defining the generalized covariance map.
Let $\Psi_{\D}$ be the inverse of $\Phi_{\D}$ as descriped in \Cref{Prop:CutToCorr}. Since $\phi_{\Delta}$ is an isomorphism of $\Corr(\Delta)$ and $\Gcut(\Delta)$ that extends to an automorphism of $\R^{\bar\D}$, the inequality $\bfp \cdot \bfx \leq p_0$ is  facet-defining  for $\Gcut(\Delta)$ if and only if  $(\bfp \Phi_{\Delta}) \cdot \bfy \leq p_0$ is facet-defining for $\Corr(\Delta)$. So we can use the switching operation on $\Gcut(\Delta)$ to uncover the following symmetries among the facets of $\Corr(\Delta)$.

\begin{cor}\label{Cor:CorrSwitching}
Let $\bfq \cdot \bfy \leq q_0$ be facet-defining for $\Corr(\Delta)$. Fix $I \subset [n]$. Let $\Phi_{\Delta}$ be the matrix representing the generalized covariance map and let $\Psi_{\Delta}$ be its inverse. Define the linear functional $\bq^{[I]} := (\bfq \Psi_{\Delta})^{(I)} \Phi_{\Delta}$. Then $\bfq^{[I]} \by \leq q_0 - (\bfq \Psi_{\Delta}) \bd^I$ is facet-defining for $\Corr(\Delta)$.
\end{cor}

We can also study the switching operation on \(\Marg(\Delta)\) by means of the maps between \(\Marg(\Delta),\Corr(\Delta)\), and \(\Gcut(\Delta)\) given in the appendix. Take the linear map $\Omega_{\Delta}$ from Proposition \ref{Prop:MargToCorr} sending $\Marg(\Delta)$ to $\Corr(\Delta)$.
As in the analysis of \(\Corr(\Delta)\), since \(\Omega_\Delta\) is an isomorphism of polytopes, if \(\bq \cdot \bfy \leq q_0\) is a facet-defining inequality for \(\Corr(\Delta)\), then \((\bq \Omega_\Delta) \cdot \bfz \leq q_0\) is a facet-defining inequality for \(\Marg(\Delta)\). Note that \(\Omega_\Delta\) is only an isomorphism when restricted to the affine hull of \(\Marg(\Delta)\). In fact, \(\Omega_\Delta\) as a linear map of vector spaces has non-trivial kernel. Therefore, it may be the case that for some facet-defining inequality \(\bfr \cdot \bfz \leq r_0\) of \(\Marg(\Delta)\) there is no \(\bq \in \R^{\overline{\Delta}}\) so that \(\bfr = \bq \Omega_\Delta\).

Let $M_0$ denote the translation of the affine hull of $\Marg(\Delta)$ to the origin. If we wish to apply the switching operation to the face $F = \{ \bfz \in \Marg(\Delta) \;:\; \bfr \cdot \bfz \leq r_0\}$ of $\Marg(\Delta)$, then we must replace the linear functional $\bfr$ defining the hyperplane \(\bfr \cdot \bfz = r_0\) with its projection $\bar\bfr$ onto $M_0$.

We denote this new hyperplane by \(\overline{\bfr} \cdot \bfz = \overline{r_0}\). Note that it is still a supporting hyperplane of \(F\), and in fact, it is the unique hyperplane which passes through \(F\) and is orthogonal to \(M_0\). 
Since $\Omega_{\D}$ is an isomorphism of $M_0$ and $\R^{\bar\D}$, there exists a $\bfq \in \R^{\bar\D}$ such that $\bar\bfr = \bfq \Omega_{\Delta}$. Let $\Pi_{\Delta}$ be the pseudoinverse of $\Omega_{\Delta}$ given in \Cref{Prop:CorrToMarg}.

\begin{cor}\label{cor:MargSwitching}
Suppose that \(\bfr \cdot \bfz \leq r_0\) is a facet-defining inequality for \(\Marg(\Delta)\) and \(\bfr\) lies in the row space of \(\Omega_{\Delta}\). Fix \(I \subset [n]\) and define the linear functional $\bfr^{\langle I \rangle} := (\bfr \Pi_\Delta \Psi_\Delta)^{(I)}\Phi_\Delta \Omega_\Delta$
and the scalar $r_0^{\langle I \rangle} := r_0  - (\bfr \Pi_\Delta \Psi_\Delta)\cdot \bd^I.$
Then \(\bfr^{\langle I \rangle}\cdot \bfz \leq r_0^{\langle I \rangle}\) is a facet-defining inequality for \(\Marg(\Delta)\).
\end{cor}
\begin{proof}
Suppose that \(\bfr \cdot \bfz \leq r_0\) is a facet-defining inequality of \(\Marg(\Delta)\) so that \(\bfr\) lies in the row space of \(\Omega\). We have that \(\bfz  = \Pi_\Delta \Psi_\Delta \Phi_\Delta \Omega_\Delta \bfz + \bfu^\emptyset\). Note that since $\bfr \in \mathrm{rowspan}(\Omega_{\Delta})$,
we have $\bfr \cdot \bfu^{\emptyset} = 0$ by definition of $\Omega_{\Delta}$. If we set \(\bfp = \bfr \Pi_\Delta \Psi_\Delta\), \(\bfx = \Phi_\Delta \Omega_\Delta \bfz\), and \(p_0 = r_0 - \bfr \cdot \bfu^\emptyset\), then \(\bfp \cdot \bfx \leq p_0\) is a facet-defining inequality of \(\Gcut(\Delta)\). Then 
\[
    \bfp^{(I)} \cdot \bfx = (\bfr \Pi_\Delta\Psi_\Delta)^{(I)}\Phi_\Delta\Omega_\Delta \bfx 
     = \bfr^{\langle I \rangle} \cdot \bfz
\]
and 
\[
    p_0^{(I)} = p_0 - \bfp \cdot \bd^I 
            = r_0  - (\bfr \Pi_\Delta \Psi_\Delta)\cdot \bd^I.
\]
Thus, \(\bfr^{\langle I \rangle} \cdot \bfz \leq r_0^{\langle I \rangle}\) is a facet-defining inequality for \(\Marg(\Delta)\).
\end{proof}

The analogue of switching for correlation and binary marginal polytopes is quite complicated to describe in terms of the coordinates of the facet-defining linear functionals. However, it is easy to determine the sets of vertices that lie on these ``switched" facets due to  \cite[Lemma 26.3.3]{cutsDeza1997} and the isomorphisms from $\Marg(\Delta)$ and $\Corr(\D)$ to  $\Gcut(\Delta)$ that send $\bu^S$  and $\bv^S$ to $\bd^S$.

\begin{cor}
Let $\mathcal{S}$ be a subset of $2^{[n]}$ such that $\conv\{ \bu^S \mid S \in \mathcal{S}\}$ is a face of $\Marg(\Delta)$ (resp. $\Corr(\D)$). Let $I \subset [n]$. Then $\conv \{ \bu^{S \triangle I} \mid S \in \mathcal{S} \}$ is also a face of $\Marg(\Delta)$ (resp. $\Corr(\D)$).
\end{cor}

The following examples compute the $\cH$-description for the generalized cut polytope of a simplex and the disjoint union of two simplices.
\begin{prop}\label{Ex:HSimplex}
The generalized cut polytope for the simplex $2^{[n]}$ has $\cH$-description   \begin{align}
  \label{eq:Hsimplex}
\switch_{2^{[n]}}\left(\sum_{\emptyset\neq F\subset [n]}x_{F}\leq 2^{n-1}\right).
 \end{align}
\end{prop}
\begin{proof}
The binary marginal polytope for $\D=2^{[n]}$ is the standard simplex in $\mathbb{R}^{2^{[n]}}$.
Since  $\Gcut(2^{[n]})$ is isomorphic to $\Marg(2^{[n]})$, it has  $2^n$ facets, and each facet contains $2^{n}-1$ vertices.  
The  vertices $\{\bfd^{S}~|~  S\subset [n]\}$ of the generalized cut polytope satisfy inequalities 
 \[\sum_{\emptyset\neq F\subset [n]}d^{\{\emptyset\}}_{F}=0~ \text{ and } ~ \sum_{\emptyset\neq F\subset [n]}d^S_{F}=2^{n-1},~ for ~\emptyset\neq S\subset [n], \]
which implies that 
\begin{align}
\label{eq:simplex1}
\sum_{\emptyset\neq F\subset [n]}x_{F}\leq 2^{n-1}    
\end{align}
is one of the facet-defining inequalities of $\Gcut(\D)$. 
Take the inequalities obtained by the switching operation on \Cref{eq:simplex1} with respect to any subset $I$ of $[n]$. There are $2^n$ of them. They are of the form 
\begin{align}
 \label{eq:hsimplex}
\sum_{\emptyset\neq F\subset [n]}(-1)^{\#(F\cap I)}x_{F}\leq 2^{n-1}-\mathbf{1}\cdot \bfd^I,~ I\subset [n],
\end{align}
where $\mathbf{1}$ is the row vector with all entries one in $\R^{\bar\D}$.  By \cite[Lemma~26.3.3]{cutsDeza1997}, the switching of \Cref{eq:hsimplex} with respect to $S$ yields the facet-defining inequality for the co-facet $\{\bd^S\}$.
Therefore,  \Cref{eq:Hsimplex} is the $\cH$-representation for $\Gcut(2^{[n]})$.  These inequalities can be written as 
\begin{align*}
&\sum_{\emptyset\neq F\subset [n]}x_{\emptyset \sqcup F}\leq 2^{n-1} \text{ and }
\sum_{\emptyset\neq F\subset [n]}(-1)^{\#(F\cap S)}x_{\emptyset \sqcup F}\leq 0, \text{ for } \emptyset\neq S\subset [n]. \qedhere
\end{align*}
\end{proof}

Let  $\Gcut(\D)$ and $\Gcut(\Gamma)$ be the generalized cut polytopes for simplicial complexes $\D$ and $\Gamma$ with disjoint ground sets. The product $\Gcut(\D)\times \Gcut(\Gamma)$ is  generalized cut polytope for their disjoint union $\Gcut(\D\sqcup\Gamma)$, since vertices of the second one are all vectors of the form $[\mathbf{d}_{\D}^S~~\mathbf{d}_{\Gamma}^T]^{tr}$ for any vertex $\mathbf{d}_{\D}^S$ of $\Gcut(\D)$ and $\mathbf{d}_{\Gamma}^T$ of $\Gcut(\Gamma)$. Hence, given   $A\bfx_{\bar\D}\leq \bfa$ and $B\bfx_{\bar\Gamma}\leq \bfb$ be the half-space descriptions for $\Gcut(\D)$ and $\Gcut(\Gamma)$, respectively, $\Gcut(\D\sqcup\Gamma)$ has the $\cH$-description
\begin{align}
\label{eq:twoSimplex}
    \begin{bmatrix} 
    A& \mathbf{0}\\
    \mathbf{0}&B
    \end{bmatrix} \begin{bmatrix}
    \bfx_{\bar\D}\\
    \bfx_{\bar\Gamma}
    \end{bmatrix}\leq \begin{bmatrix}
    \bfa\\
     \bfb
    \end{bmatrix}.
\end{align}

\begin{ex}\label{Ex:HDisjointUnionSimplices} The simplicial complex obtained by taking the disjoint union $2^{[m]} \sqcup 2^{[n]}$ is of particular interest since it is one key building block for unimodular simplicial complexes \cite{unimodHierModels2017}. \Cref{Ex:HSimplex} induces the following   $\cH$-description of the polytope $\Gcut(2^{[m]}\sqcup 2^{[n]})$:
\begin{align}
\label{eq:halfspacesOfDisjoinUnion} 
\switch_{2^{[m]}}(\sum_{\emptyset\neq F\subset [m]}x_{F \sqcup \emptyset} \leq 2^{m-1})\cup \switch_{2^{[n]}}(\sum_{\emptyset\neq F\subset [n]}x_{\emptyset \sqcup F}\leq 2^{n-1}).
 \end{align}
 
\end{ex}

\section{$\cH$-Representation for Turtle Complexes}
\label{section:turtleComplexes}

In this section, we give a complete $\cH$-representation for the generalized cut polytopes of a family of simplicial complexes which we call turtle complexes. They are named turtle complexes, because they are part of the ``shell" of the simplex.

\begin{defn}
\label{def:turtle}
A \emph{turtle complex} on ground set $[n]$ is a simplicial complex on $[n]$ and facets of cardinality $n-1$. Define  \(\mathbb{T}_n^k\) to be the turtle complex on $[n]$ with facets \(F_i = [n]\setminus\{i\}\), for $i=1,\dots,k$. 
\end{defn}
Up to ordering on the vertices on the ground  set, any turtle complex is a \(\mathbb{T}_n^k\). Examples of turtle complexes are the boundary of the $(n-1)$-simplex (\(\mathbb{T}_n^n\)) and the simplicial complex in Example \ref{Ex:RunningCut} (\(\mathbb{T}_4^2\)).

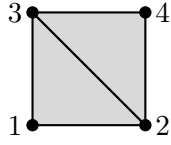
\begin{figure}
    \centering
    \begin{tikzpicture}[scale=1.5]         
    \draw[thick, fill = gray!30] (0,0) -- (1,0)-- (0,1)--(0,0);
    \draw[thick,fill=gray!30] (1,0)--(1,1)--(0,1)--(1,0);
    \draw[fill] (0,0) circle  [radius=.05];
    \node[left] at (0,0) {1};
    \draw[fill] (1,0) circle [radius=.05];
    \node[right] at (1,0) {2};
    \draw[fill] (0,1) circle  [radius=.05];
    \node[left] at (0,1) {3};
    \draw[fill] (1,1) circle  [radius=.05];
    \node[right] at (1,1) {4};
    \end{tikzpicture}
    \caption{The turtle complex used in Example \ref{Ex:RunningCut}.}
    \label{fig:turtle}
\end{figure}

\begin{defn}
Let $S \subset [n]$. We define the linear functional $\bc^S$ on $\R^{\overline{\mathbb{T}_n^k}}$ by
\[
c^S_F = 
\begin{cases}
0 & \text{ if } \#(S \cap F) \text{ is odd} \\
1 & \text{ if } \#(S \cap F) \text{ is even.}
\end{cases}
\]
\end{defn}

Let \(\mathcal{I} = \cap_{i=1}^k F_i\) be the intersection of all the facets, and throughout this section, assume that \(F_i = [n] \setminus \{i\}\).  In this section, we will prove the following theorem.

\begin{thm}
\label{thm:turtle}
Let \(\mathcal{J} \subset 2^{[n]}\) be the set of all subsets of $[n]$ of even cardinality.  The polytope \(\Gcut(\mathbb{T}_n^k)\) has $\cH$-description
\(
   \switch_\mathcal{J}(\bc^S \cdot \bfx \leq 2^{n-2}~|~S\subset [n], S\cap \mathcal{I} = \emptyset, \#S \text{ is odd}).
\)
\end{thm}

Our outline is as follows. We compute the Gale transform of \(\Gcut(\mathbb{T}_n^k)\) and use Gale duality to determine what vertices lie on each co-facet of the polytope. For the facets which do not contain \(\bd^\emptyset\), we will show that their supporting hyperplanes are \(\bc^S \cdot \bfx = 2^{n-2}\). Finally, we will use the switching operation to get the facets that do contain \(\bd^\emptyset\).

\begin{lemma}\label{turtle-gale-transform}
Let \(\overline{D}_\Delta\) be the \((2^{n} - 2^{\#\mathcal{I}}) \times 2^{n}\) matrix whose columns are given by \(\begin{bmatrix} 1\\ \bfd^S \end{bmatrix}\). Then the kernel of \(\overline{D}_\Delta\) has a basis of vectors \(\bal^T \in \mathbb{R}^{2^{[n]}}\)  where
    \[
    \alpha_A^T = 
    \begin{cases}
        (-1)^{\#A}  & A \cap \mathcal{I} = T \\
        0           & \mbox{otherwise},
    \end{cases}
    \]
fore each \(A\subset [n]\) and \(T \subset \mathcal{I}\).
\end{lemma}

\begin{proof}
\(\Gcut(\mathbb{T}_n^k)\) is full-dimensional, so the kernel of \(\overline{D}_\Delta\) has dimension \(2^{\#\mathcal{I}}\). Since the vectors are indexed by subsets of \(\mathcal{I}\), there are \(2^{\#\mathcal{I}}\) of them. To show they are a basis it suffices to show that they are in the kernel and linearly independent. 

To see that they are linearly independent, suppose that we have a linear combination
    \[
    \sum_{T \subset \mathcal{I}} c_T \bal^T = 0. 
    \]
That for any two distinct subsets \(T, T'\) of  \(\mathcal{I}\), \(\alpha_{T'}^T\) is 0 since \(T' \cap \mathcal{I} = T'\) different from \(T\). Hence,
    \[
    \left(\sum_{T \subset \mathcal{I}} c_T \bal^T \right)_{T'} = (-1)^{\#T'}c_{T'} = 0.
    \]
Hence, all  \(c_{T'}\) are zero, for  \(T' \subset \mathcal{I}\), which concludes that \(\{\bal^T\}_{T\subset \mathcal{I}}\) is linearly independent. 

It remains to show that \(\bal^T \in \ker(\overline{D}_\Delta)\). For \(F \in \overline{\mathbb{T}_n^k}\) we have that
\begin{align*}
(V\bal^T)_F &= \sum_{A \subset [n]} d_F^A \alpha_A^T 
              = \sum_{\substack{A \cap \mathcal{I} = T \\ \#(F \cap A) ~is ~odd}} (-1)^{\#A}.
\end{align*}
We can see that this sum is zero by using a sign reversing involution on the set
    \[
    \mathcal{A}^T_F =  \{A \subset [n] ~|~ A \cap \mathcal{I} = T, ~ \#(F \cap A) \mbox{ is odd}\}.
    \]
Since \(F\in \mathbb{T}_n^k\), there exists some \(\ell\) in \([k]\) so that \(F \subset F_\ell\). In particular \(\ell\) is in neither \(F\) nor \(\mathcal{I}\). Define $\sigma_\ell$ by \(\sigma_\ell(A) = A \triangle \{\ell\}\) for each \(A \in \mathcal{A}^T_F\). Since \(\ell\) is not in \(\mathcal{I}\), we have
\(
    A \cap \mathcal{I} = \sigma_\ell(A) \cap \mathcal{I}.
\)
Since \(\ell\) is not in \(F\), we have
\(
    \#(A \cap F) = \#(\sigma_\ell(A) \cap F).
\)
We have shown that if \(A \in \mathcal{A}^T_F\), then \(\sigma_\ell(A) \in \mathcal{A}^T_F\), and \(\sigma_\ell\) is easily seen to be an involution on \(\mathcal{A}^T_F\). Finally, if we set the sign of \(A\in \mathcal{A}^T_F\) to be \((-1)^{\#A}\), then \(\sigma_\ell\) is also sign-reversing. It follows that \((\overline{D}_\Delta \bal^T)_F = 0\) for \(F \in \overline{\mathbb{T}_n^k}\). 

The last thing we need to show is that the first entry of \(\overline{D}_\Delta \bal^T\) is zero. The first entry is given by the dot product of the all 1's vector and \(\bal^T\) which is the sum
\(
    \sum_{A\cap \mathcal{I} = T} (-1)^{\#A}.
\)
This is zero by using \(\sigma_\ell\) for any \(\ell \notin \mathcal{I}\).
\end{proof}

The next lemma is a straightforward consequence of Lemma \ref{turtle-gale-transform}.

\begin{lemma}\label{turtle-cofacets}
We denote the standard basis of \(\mathbb{R}^\mathcal{I}\) by \(\{\bfe_T\}_{T \subset \mathcal{I}}\). The Gale transform of \(\mathbb{T}_n^k\) as in \Cref{turtle-gale-transform} is \(\mathcal{B} = \{\bb_S ~|~ S \subset [n]\}\), where 
\[
\bb_S = \sum_{\substack{T \subset \mathcal{I} \\ S\cap \mathcal{I} = T}} (-1)^{\#S}\bfe_T = (-1)^{\#S} \bfe_{S\cap \mathcal{I}}.
\]
Therefore, the circuits of \(\mathcal{B}\) are all of the form \(\{\bb_S, \bb_{S'}\}\) where \(S\cap \mathcal{I} = S' \cap \mathcal{I}\) and \(\#S\) and \(\#S'\) have opposite parity. The co-facets of \(\Gcut(\mathbb{T}^k_n)\) are \(\{\bd^S, \bd^{S'}\}\), where \(S \cap \mathcal{I} = S' \cap \mathcal{I}\), \(\#S\) is even, and \(\#S'\) is odd.
\end{lemma}

At this point, we have the tools to give the supporting hyperplanes for each facet which does not contain \(\bd^\emptyset\).

\begin{lemma}\label{turtle-halfs}
The facet with co-facet \(\{\bd^\emptyset,\bd^S\}\) where \(S \cap \mathcal{I} = \emptyset\) and \(\#S\) is odd has supporting hyperplane $\bc^S \cdot \bfx = 2^{n-2}.$
\end{lemma}

\begin{proof}
First, note that \(\bc^S \cdot \bd^\emptyset = 0\) since \(\bd^\emptyset = {\bf 0}\), and \(\bc^S \cdot \bd^S = 0\) by definition of \(\bc^S\). Let \(T \subset [n]\). Note that 
    \begin{equation}\label{eqn:vectorspacecount}
    \bc^S \cdot \bd^T = \#\{A \in \overline{\mathbb{T}_n^k} ~|~ \#(A \cap S) \text{ is even and } \#(A \cap T) \text{ is odd}\}.
    \end{equation}
We will show that \(\bc^S \cdot \bd^T = 2^{n-2}\) by finding a codimension 2 affine subspace of \(\mathbb{F}_2^{n}\) which counts the set on the right-hand side of equation \ref{eqn:vectorspacecount}. Define
    \[
    H_{S,T} = \left\{ \bfi^A \in \mathbb{F}_2^{n} ~|~ M_{S,T} \cdot \bfi^A = \begin{pmatrix} 0 \\ 1 \end{pmatrix}\right\},
    \]
where \(M_{S,T}\) is the matrix whose first row is the indicator vector \(\bfi^S\) and the second row is the indicator vector \(\bfi^T\). Since \(S \neq \emptyset\) and \(\emptyset\neq T \neq S\), this matrix has rank 2. Hence, the cardinality of \(H_{S,T}\) is equal to \(2^{n-2}\).

The subspace \(H_{S,T}\) can be counted by subsets of \([n]\) whose intersection with \(S\) has even cardinality and whose intersection with \(T\) has odd cardinality. Thus, in order to show that \(\bc^S \cdot \bd^T = 2^{n-2}\), we only need to show that every vector in \(H_{S,T}\) corresponds to a subset which is an element of \(\overline{\mathbb{T}_n^k}\). Let \(A\) be a subset of \([n]\) which is not in \(\mathbb{T}_n^k\). We argue that the indicator vector for \(A\) is not in \(H_{S,T}\). Note \({\bf 0} \notin H_{S,T}\), so we may assume that \(A\) is non-empty. We know that \(S \cap [k] = S\) since \(S \cap \mathcal{I} = \emptyset\) and \(\mathcal{I} = [n]\setminus[k]\), and we know that \(\#S\) is odd. If \(A\) is not in \(\mathbb{T}_n^k\), then it must be of the form 
\(
A = [k] \sqcup B,
\)
where \(B \subset \mathcal{I}\). It follows that if \(A\) is not in \(\mathbb{T}_n^k\) then \(\#(A\cap S) = \#(S \cap [k])\) is odd. This means that \(\bfi^A\) is not in \(H_{S,T}\). Equivalently, if \(\bfi^A\) is in \(H_{S,T}\), then \(A\) is in \(\overline{\mathbb{T}_n^k}\). So 
    \[
    2^{n-2} = \#H_{S,T} = \#\{A \in \overline{\mathbb{T}_n^k} ~|~ \#(A \cap S) \text{ is even and } \#(A \cap T) \text{ is odd}\} = \bc^S \cdot \bd^T. \qedhere
    \]
\end{proof}

\begin{proof}[Proof of Theorem \ref{thm:turtle}]
Consider any facet \(F\). By Lemma \ref{turtle-cofacets}, the corresponding co-facet is \(\{\bd^S,\bd^{S'}\}\) where \(\#S\) is even, \(\#S'\) is odd, and \(S\cap \mathcal{I} = S'\cap\mathcal{I}\). Let \(F'\) be the facet with co-facet \(\{\bd^\emptyset, \bd^{S\triangle S'}\}\).  By Remark \ref{mod-switching}, the facet-defining inequality for \(F\) can be found by switching the facet-defining inequality of \(F'\) by \(S\). We have shown that any facet-defining inequality for \(\Gcut(\mathbb{T}_n^k)\) is be obtained by switching an inequality as in Lemma \ref{turtle-halfs} by an even subset. Thus, the facet-defining inequalities for \(\Gcut(\mathbb{T}_n^k)\) are exactly as claimed in the statement of the theorem. 
\end{proof}

\section{$\mathcal{H}$-Representation for Cones}

This section shows how one can derive $\cH$-representations for the generalized cut polytope of a cone over a simplicial complex $\D$ using the $\cH$-representation of the generalized cut polytope of $\D$.

\begin{defn}
Let $\D$ be a simplical complex on ground set $[n]$.  We define $\cC(\D)$ to be the simplicial complex on ground set $[n]\cup \{\ell\}$ with  facets
\[
\facet(\cC(\D))=\{F\cup \{\ell\}~|~ F\in \facet(\D)\}.
\]
\end{defn}

The  notation  $\bar{\D}\sqcup \{\ell\}$ will denote $\{F\cup \{\ell\}~|~ F\in \bar{\D}\}$. The notations $\bfx_{\bar{\D}}$ and $\bfx_{\bar{\D}\sqcup \{\ell\}}$ are used to denote the vector of variables indexed by the elements of $\bar{\D}$ and $\bar{\D}\sqcup \{ \ell \}$, respectively.  We will use $\mathbf{0}$ to denote the vector of all zero entries, and  $\mathbf{1}$ the vector of all one entries.  
\begin{thm}
\label{thm:coneHalfspace}
Let  $A\bfx_{\bar{\D}}\leq \bfc$ be the  half-space description for the generalized cut polytope of a simplicial complex  $\D$. The generalized cut polytope $\Gcut(\cC(\D))$ for the cone of $\D$ has half-space description 
\begin{align*}
    \begin{bmatrix} 
    A & A & 1\\
    A & -A & -1
    \end{bmatrix} \begin{bmatrix}
    \bfx_{\bar{\D}}\\
    \bfx_{\bar{\D}\sqcup \{\ell\}}\\
    x_{\{\ell\}}
    \end{bmatrix}\leq \begin{bmatrix}
    2\bfc\\
     \mathbf{0}
    \end{bmatrix}.
\end{align*}
\end{thm}

\begin{proof}

Let $D$ be the matrix with the vertices of $\Gcut(\D)$ as its column vectors. The vertices of $\Gcut(\cC(\D))$ are $\bfd_{\cC}^{S}=\begin{bmatrix}\bfd^S \\ \bfd^S \\ 0
\end{bmatrix}$ and
$\bfd_{\cC}^{S\cup \{\ell\}}=\begin{bmatrix}\bfd^S \\ \mathbf{1}-\bfd^S \\ 1
\end{bmatrix}$
in $\mathbb{R}^{\bar{\D}}\times \mathbb{R}^{\bar{\D}\sqcup \{\ell\}}\times \mathbb{R}^{\{\ell\}}$
where $\mathbf{1}$ is the vector of all ones. The homogenized matrix of $\Gcut(\cC(\D))$ has form
\begin{align*}
\Bar{D}_{\cC}=\begin{bmatrix}
\textbf{1}&\textbf{1}\\
D & D\\
D & \textbf{1}-D\\
\textbf{0}& \textbf{1}
\end{bmatrix}.
\end{align*}
A vector $\begin{bmatrix}\bft &\bft_{\{\ell\}}
\end{bmatrix}^{tr}\in \R^{2^{[n]}}\times \R^{2^{[n]}\sqcup p}$ in the kernel of  $\Bar{D}_{\cC}$ has the set of conditions: 
\begin{align*}
   & \sum\limits_{S\subset [n]}(t_{S}+t_{S\cup\{\ell\}})=0, & \qquad & 
   \sum\limits_{S\subset [n]}d_F^{S}(t_{S}+t_{S\cup\{\ell\}})=0 & \quad & \text{for each } F\in\bar{\D},\\
  &\sum\limits_{S\subset [n]}t_{S\cup\{\ell\}}= 0,
  & \qquad &
  \sum\limits_{S\subset [n]}d_F^{S}t^{S}+(1-d_F^{S})t_{S\cup\{\ell\}})=0 & \quad & \text{ for each } ~F\in\bar{\D},~~~~~
\end{align*}
which are equivalent to the following set of conditions:
\begin{align*}
  &\sum\limits_{S\subset [n]}t_{S}=0, & \qquad &
   \sum\limits_{S\subset [n]}d_F^{S}t_{S}=0&  \text{for each } F\in\bar{\D},\\
   &\sum\limits_{S\subset [n]}t_{S\cup\{\ell\}}=0,
   & \qquad & 
   \sum\limits_{S\subset [n]}d_F^{S}t_{S\cup\{\ell\}}=0 &  \text{ for each } F\in\bar{\D}.
\end{align*}
The first two equations are conditions on $\bft$ to be in the kernel of the homogenized matrix  $\bar{D}$. The next two equations are conditions for $\bft_{\{\ell\}}$ to be in the kernel of  $\bar{D}$. 
Let  $B_1,\dots, B_{i}\in \R^{2^{[n]}}$, where $i=2^n-\#\D$ be a basis for the kernel of $\bar{D}$, and let $\mathcal{B}=\{\bfb_S~| ~ S\subset [n]\}$ be the Gale transform for $\Gcut(\D)$ with respect to this basis, i.e. the rows of the matrix  $K=[B^1~\dots ~B^i]$. The columns  of the matrix
\[K_{\cC}=\begin{bmatrix}
B^1,\dots, B^{i} & \mathbf{0},\dots, \mathbf{0}\\
\mathbf{0},\dots, \mathbf{0} & B^1,\dots, B^{i}
\end{bmatrix}\] 
form a basis for the kernel of $\Bar{D}_{\cC}$. The Gale transform for $\Gcut(\cC(\D))$ with respect to this basis is the set of vectors
\[\mathcal{B}_{\cC}=\{\bfb_S^{\cC}=\begin{bmatrix}
\bfb_{S}\\\mathbf{0}\end{bmatrix},\bfb_{S\cup\{\ell\}}^{\cC}=\begin{bmatrix}\mathbf{0}\\\bfb_{S}\end{bmatrix}~|~ S\subset [n]\}.\]
Take an arbitrary minimal co-face for $\Gcut(\cC(\D))$. Such a set is a collection of vertices  $\{\bfd_{\cC}^S~|~S\in \mathcal{S}\subset 2^{[n]}\}\cup \{\bfd_{\cC}^{T\cup\{\ell\}}~|~T\in \mathcal{T}\subset 2^{[n]}\}$ of $\Gcut(\cC(\D))$. By the construction of the Gale dual for $\Gcut(\cC(\D))$,  zero is in relative interior of the convex hull of the respective vectors of the Gale dual if and only if  zero is in the relative interior of  $\{\bfb^{\cC}_S~|S\in \mathcal{S}\subset 2^{[n]}\}$ and of vertices $\{\bfb^{\cC}_{T}~|~T\in \mathcal{T}\subset 2^{[n]}\}$ simultanuesly. Hence, all the minimal co-faces for $\Gcut(\cC(\D))$ are of the form $\{\bfd_{\cC}^S~| S\in \mathcal{S}\subset 2^{[n]}\}$ and 
 $\{\bfd_{\cC}^{S\cup \{\ell\}}~| S\in \mathcal{S}\subset 2^{[n]}\}$, where $\{\bfd^S~| S\in \mathcal{S}\subset 2^{[n]}\}$ is a minimal co-face for $\Gcut(\D)$. 

Lastly, one can easily check that the  inequality $\bfa\bfx_{\bar{\D}}\leq \bfc$ associated to the minimal co-face $\{\bfd^{S}~|~ S\in \mathcal{S}\subset [n]\}$ of $\Gcut(\D)$ induces the inequalities $\bfa\bfx_{\bar{\D}}+\bfa\bfx_{\bar{\D}\sqcup \{\ell\}}+x_{\{\ell\}}\leq 2\bfc$ for the minimal co-face $\{\bfd_{\cC}^{S}~|~ S\in \mathcal{S}\subset [n]\}$ and $\bfa\bfx_{\bar{\D}}-\bfa\bfx_{\bar{\D}\sqcup \{\ell\}}-x_{\{\ell\}}\leq 0$ for the minimal co-face $\{\bfd_{\cC}^{S\cup\{\ell\}}~|~ S\in \mathcal{S}\subset [n]\}$ of $\Gcut(\cC^{\ell}(\D))$. Notice that the second inequality is a switching of the first by the set  $I=\{\ell\}$.
\end{proof}


\begin{ex}
\Cref{thm:coneHalfspace} and \Cref{eq:halfspacesOfDisjoinUnion} induce the following $\cH$-representation for the cone over the disjoint union of two simplices  $2^{[m]}\sqcup2^{[n]}$:
\begin{equation*} \label{eq:conetwosimplexes} 
\begin{cases}
\switch_{2^{[m]\cup \{\ell\}}}(\sum\limits_{\emptyset \neq S\subset [m]\}}(y_S + y_{S\cup \{\ell\}}) +y_{\{\ell\}} \leq 2^{m})\\
 \switch_{2^{[n]\cup \{\ell\}}}(\sum\limits_{\emptyset\neq S\subset [n]}(y_S + y_{S\cup \{\ell\}}) +y_{\{\ell\}} \leq 2^{n}).
\end{cases}
\end{equation*}
\end{ex}

From the perspective of marginal polytopes, $\Marg(\cC(\D))$ is a sub-direct product of $\Marg(\D)$ with itself as its vertices are of the form 
\(\begin{bmatrix}\bfu^S \\
\mathbf{0}
\end{bmatrix}\), and  \(\begin{bmatrix}\mathbf{0}\\
\bfu^S
\end{bmatrix},\) where $\bfu^S$ are vertices of $\Corr(\D)$. This way, one can deduce information about the facets of $\Marg(\cC(\D))$ \cite[Section~15.1.3]{handbook2017}, but a half-space description needs still to be found. 

One can extend inductively results about cones to $k-$cone over $\D$, denoted $\cC^k(\D)$, defined as the simplicial complex over $[n]\cup \{\ell_1,\dots,\ell_k\}$ with facets
\[
\facet(\cC^k(\D))=\{F\cup \{\ell_1,\dots,\ell_k\}~|~ F\in \facet(\D)\}.
\]
Such an example is the simplicial complex in \Cref{fig:turtle}; it is a $2$-cone with  $\D=\{\{1\},\{4\}\}$ and $\Bell=\{2,3\}$.
The observation  that $k$-cones are obtained by taking iteratively $k$ times the cone over the simplicial complex $\D$, i.e. 
\begin{align}
\label{eq:cone}
\cC^{k}(\D)=\underbrace{\cC(\cC(\dots \cC}_{k-times}(\D))\dots)
\end{align}
induces \Cref{cor:kCones} of \Cref{thm:coneHalfspace} on their $\cH$-descriptions.

\begin{cor}
\label{cor:kCones}
If  $A\bfx_{\bar{\D}}\leq \bfc$ is the $\cH$-description for $\Gcut(\D)$, the polytope $\Gcut(\cC^k(\D))$ has $\cH$-description 
\begin{align*} 
    \begin{bmatrix} 
    A_{\cC^{k-1}} & A_{\cC^{k-1}} & \mathbf{1}\\
    A_{\cC^{k-1}} & -A_{\cC^{k-1}} & -\mathbf{1}
    \end{bmatrix} \begin{bmatrix}
    \bfx_{\Bar{\cC^{k-1}(\D)}}\\
     \bfx_{\Bar{\cC^{k-1}(\D)}\sqcup{\{\ell_k\}}}\\
      \bfx_{\{\ell_k\}}
    \end{bmatrix}\leq \begin{bmatrix}
    2^{k}\bfc\\
     \mathbf{0}
    \end{bmatrix},
\end{align*}
where $A_{\cC^{k-1}}$ is the coefficient matrix in the $\cH$-description of $\Gcut(\cC^{k-1}(\D))$. 
\end{cor}

\begin{obs}
\label{obs:turtleCones}
Take $k<n$ to be two natural numbers. The simplicial complex $\cC^{n-k}(2^{[k]}\setminus [k])$ has facets 
\(
\{[k]\setminus \{i\} \cup \{k+1,\dots,n\}~|~ i=1,2\dots,k\}= \{[n]\setminus \{i\}~|~ i=1,2\dots,k\},
\)
which are exactly facets of $\mathbb{T}^k_n$. This observation allows one to use \Cref{cor:kCones} for the $\cH$-representation of turtle complexes that are not the boundary of a simplex. Carrying out the computations, one gets the same results as in \Cref{section:turtleComplexes}. Notice that the boundary of a simplex  $\mathbb{T}^n_n$ is not a cone over some smaller simplicial complex. 
\end{obs}

\section{$\mathcal{H}$-Representation  for the Alexander Dual of the Disjoint Union of two Simplexes}
\label{sec:alexDual}

In \cite{unimodHierModels2017}, Bernstein and Sullivant characterized
the simplicial complexes whose binary hierarchical models
have unimodular design matrices.
Such a simplicial complex is called \emph{unimodular}.
They show that every unimodular simplicial complex can be obtained by cone,
Lawrence lift or ghost vertex operations to $2^{[m]}$,
$2^{[m]} \sqcup 2^{[n]}$, or the Alexander dual of $2^{[m]} \sqcup 2^{[n]}$.

\begin{defn}
For a simplicial complex $\Delta$ on $[n]$, the \emph{Alexander dual}
of $\Delta$, denoted $\Delta^*$ is the simplicial complex with faces $\{S \subset [n] \mid [n] \setminus S \not\in \Delta \}$.
\end{defn}

In the case where $\Delta = 2^{[m]} \sqcup 2^{[n]}$, we follow
the notation of \cite{unimodHierModels2017} and write 
$D_{m,n} := (2^{[n]} \sqcup 2^{[n]})^*.$
The simplicial complex $D_{m,n}$ has a nice explicit description;
indeed, $D_{m,n} = \{ S \sqcup T \mid S \subsetneq [m], T \subsetneq [n]\}$.
The characterization of unimodular simplicial complexes in \cite{unimodHierModels2017}
motivated the following investigation of $\Gcut(D_{m,n})$.

In \cite{unimodMatriod2017}, Bernstein and O'Neill gave a 
characterization of the matroid underlying the Gale dual of $\Gcut(D_{m,n}).$
The underlying oriented matroid of the Gale transform of $\Gcut(D_{m,n})$ is naturally isomorphic to the oriented graphic matroid of $K_{2^n, 2^m}$ with the following orientation. Let $A \subset [m]$ and $B \subset [n]$ be vertex labels of $K_{2^m, 2^n}$. Throughout this section, we write each element of $D_{m,n}$ as $A \sqcup B$ to indicate that $A \subset [m]$ and $B \subset [n]$. We direct the edge $A \rightarrow B$ if $A \sqcup B$ is even, and $A \leftarrow B$ if $A \sqcup B$ is odd. Adopting the notiation in  \cite{unimodMatriod2017}, we call this digraph $\dig$.

Each edge of $\dig$ corresponds to a subset of $[m] \sqcup [n]$ by taking the disjoint union of the subsets of $[m]$ and $[n]$ that label the nodes of the edge. When we wish to recall the orientation of the edge corresponding to $A \sqcup B$, we will denote it by $\overrightarrow{A \sqcup B}$ or $\overleftarrow{A \sqcup B}$. In this way, each edge of $\dig$ also corresponds to a vertex of $\Gcut(D_{m,n})$. Every directed cycle of $\dig$ corresponds to a \emph{co-facet} of $\Gcut(D_{m,n})$; that is, the set of vertices not on a given facet. So facets of $\Gcut(D_{m,n})$ are in bijection with directed cycles of $\dig$. To simplify notation, we will sometimes refer to facets and co-facets of $\Gcut(D_{m,n})$ by the sets that index their vertices.

\begin{lemma}
Every directed in cycle in $\dig$ has length divisible by $4$.
\end{lemma}

\begin{proof}
Observe that every directed cycle in $\dig$ must be of length at least 4, and must have at least one even and one odd set  both in $[m]$ and in $[n]$. Without loss of generality, let $A_1$ be an odd set in $[m]$ and let $B_1$ be an odd set in $[n]$ that are in a cycle $C$. Then,  the definition of $\dig$ induces that the directed cycle $C$ must be of the form: 
\[
C = \{ \overrightarrow{A_1 \sqcup B_1}, \overleftarrow {A_2 \sqcup B_1}, \overrightarrow{A_2 \sqcup  B_2}, \dots, \overrightarrow{A_{k} \sqcup B_{k}}, \overleftarrow{A_1 \sqcup B_k} \},
\]
where $\#A_i$ is odd if $i$ is odd and even if $i$ is even, and similarly for each $B_i$. Note that $k$ must be even since the edge $\overleftarrow{A_1 \sqcup B_k}$ is directed towards $A_1 \subset [m]$. By construction, $C$ has $2k$ edges. Since $k$ is even, the number of edges in $C$ is divisible by $4$.
\end{proof}

Every cycle $C$ in $\dig$ of length $4 k$ can be obtained by joining disjoint cycles $C_1$ of length $4$ and $C_2$ of length $4 k - 4$ in the following way. Given two directed cycles in $\dig$, one can write them as collection of directed edges as follows:
\begin{align*}
C_1 &= \{ \overrightarrow{A_1 \sqcup B_1}, \overleftarrow{A_2 \sqcup B_1}, \overrightarrow{A_2 \sqcup B_2}, \overleftarrow{A_1 \sqcup B_2}\}, \\
C_2 &= \{ \overrightarrow{A_1' \sqcup B_1'}, \overleftarrow{A_2' \sqcup B_1'}, \dots , \overrightarrow{A_{2k-2}' \sqcup B_{2k -2}'}, \overleftarrow{A_1' \sqcup B_{2k-2}'}\},
\end{align*}
where $A_1, B_1, A_1'$ and $B_1'$ all have the same parity. Then we define the \emph{cycle obtained by gluing $C_1$ and $C_2$ along $A_1 \sqcup B_1$ and $A_1'\sqcup B_1'$} by 
\[
C = (C_1 \cup C_2 \cup \{\overrightarrow{A_1 \sqcup B_1'}, \overrightarrow{A_1' \sqcup B_1} \}) \setminus \{\overrightarrow{A_1 \sqcup B_1}, \overrightarrow{A_1' \sqcup B_1'} \}.
\]
An example of this operation is depicted for $m = n =2$ in Figure \ref{Fig:CycleGluing}. In this example, we exchange the edges $\overrightarrow{\{1\} \sqcup \{3\}}$ and $\overrightarrow{\{2\} \sqcup \{4 \}}$ for the edges $\overrightarrow{\{1\} \sqcup \{4\}}$ and $\overrightarrow{\{2\} \sqcup \{3 \}}$ to obtain a cycle of length 8.

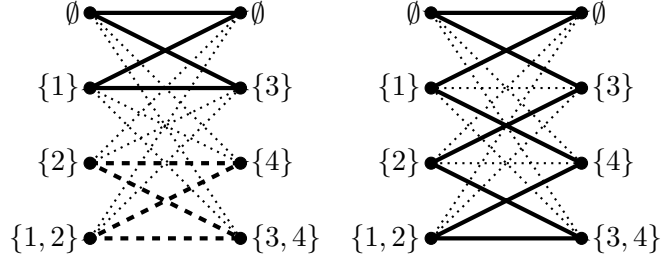
\begin{figure}
    \centering
    \begin{tikzpicture}
    \draw[fill] (0,0) circle [radius = .08];
    \node[left] at (0,3) {$\emptyset$};
    \draw[fill] (0,1) circle [radius = .08];
    \node[left] at (0,2) {$\{ 1 \}$};
    \draw[fill] (0,2) circle [radius = .08];
    \node[left] at (0,1) {$\{ 2 \}$};
    \draw[fill] (0,3) circle [radius = .08];
    \node[left] at (0,0) {$\{ 1,2 \}$};
    \draw[fill] (2,0) circle [radius = .08];
    \node[right] at (2,0) {$\{3,4\}$};
    \draw[fill] (2,1) circle [radius = .08];
    \node[right] at (2,1) {$\{4\}$};
    \draw[fill] (2,2) circle [radius = .08];
    \node[right] at (2,2) {$\{3\}$};
    \draw[fill] (2,3) circle [radius = .08];
    \node[right] at (2,3) {$\emptyset$};
    \draw[ultra thick] (0,3) -- (2,3) -- (0,2) -- (2,2) -- (0,3);
    \draw[ultra thick, dashed] (0,1) --(2,1) -- (0,0) -- (2,0) -- (0,1);
    \draw[thick, dotted] (0,0) -- (2,3) -- (0,1) -- (2,2) -- (0,0);
    \draw[thick, dotted] (0,3) -- (2,1) -- (0,2) -- (2,0) -- (0,3);
    \end{tikzpicture}
    \begin{tikzpicture}
    \draw[fill] (0,0) circle [radius = .08];
    \node[left] at (0,3) {$\emptyset$};
    \draw[fill] (0,1) circle [radius = .08];
    \node[left] at (0,2) {$\{ 1 \}$};
    \draw[fill] (0,2) circle [radius = .08];
    \node[left] at (0,1) {$\{ 2 \}$};
    \draw[fill] (0,3) circle [radius = .08];
    \node[left] at (0,0) {$\{ 1,2 \}$};
    \draw[fill] (2,0) circle [radius = .08];
    \node[right] at (2,0) {$\{3,4\}$};
    \draw[fill] (2,1) circle [radius = .08];
    \node[right] at (2,1) {$\{4\}$};
    \draw[fill] (2,2) circle [radius = .08];
    \node[right] at (2,2) {$\{3\}$};
    \draw[fill] (2,3) circle [radius = .08];
    \node[right] at (2,3) {$\emptyset$};
    \draw[ultra thick] (0,0) -- (2,0) -- (0,1) -- (2,2) -- (0,3) -- (2,3) -- (0,2) -- (2,1) -- (0,0);
    \draw[thick, dotted] (0,0) -- (2,2) -- (0,2) -- (2,0) -- (0,3) -- (2,1) -- (0,1) -- (2,3) -- (0,0);
    \end{tikzpicture}
    \caption{Two cycles in $G_2^{2,2}$, depicted in bold solid and dashed lines, and the cycle obtained by gluing them along $\{1\} \sqcup \{3\}$ and $\{2\} \sqcup \{4 \}$.}
    \label{Fig:CycleGluing}
\end{figure}

In Lemma \ref{Lem:Length4Cycles}, we first characterize the inequalities defining facets that correspond to cycles of length 4.
Then in the rest of this section, we develop machinery for computing the inequality for a cycle obtained by gluing two smaller cycles. We ultimately give an inductive facet description of $\Gcut(\Dmn)$ in Theorem \ref{Thm:ADualFacets}.

In the case of directed cycles of length $4$ in $\dig$ 
that contain the edge corresponding to $\emptyset$, 
we can directly characterize the corresponding facet
of $\Gcut(D_{m,n})$. 
Let $\O \subset [m]$ and $\U \subset [n]$ be sets of odd cardinality. 
Then $\{\emptyset, \O, \U, \O \sqcup \U\}$ 
is a co-facet of $\Gcut(D_{m,n})$. 
Let $\Ev(\O,\U)$ be the set of all sets of the form
$A \sqcup B \neq \emptyset$ such that $\#(A \cap \O)$ and
$\#(B \cap \U)$ are both even.
Note that $\Ev(\O,\U) \subset \Dmn$ since $[m] \cap \O$
and $[n] \cap \U$ both have odd cardinality.

\begin{lemma}\label{Lem:Length4Cycles}
The inequality
\begin{equation}\label{Eqn:Length4Cycles}
\sum_{S \in \Ev(\O,\U)} x_S \leq 2^{m+n-3}
\end{equation}
is the facet-defining inequality for the co-face $\{\emptyset, \O, \U, \O \sqcup \U\}$.
\end{lemma}

\begin{proof}
By construction, we have that for each $T \in \{\emptyset, \O, \U, \O \sqcup \U\}$, $\sum_{S \in \Ev(\O,\U)} d_S^T = 0.$
So each of these $\bd^T$ do not satisfy Equation (\ref{Eqn:Length4Cycles}) with equality.

We must show that for $T \not\in \{\emptyset, \O, \U, \O \sqcup \U\}$,
$\bd^T$ satisfies equality in Equation (\ref{Eqn:Length4Cycles}).
Consider the vector space $\mathbb{F}_2^{m+n}$ over the finite field of order 2.
For every $S \subset [m] \sqcup [n]$, let $\bfi^S$ be the indicator vector of $S$.
For a fixed $T \not\in \{\emptyset, \O, \U, \O \sqcup \U\}$, let $M$ be the
matrix with rows $(\bfi^{\O})^{tr}, (\bfi^{\U})^{tr}$ and $(\bfi^T)^{tr}$.
Note that $M$ has rank 3 since $T \not\in \{\emptyset, \O, \U, \O \sqcup \U\}$.
Then a set $A \subset [m] \sqcup [n]$ simultaneously is in $\Ev(\O,\U)$ and has $d^T_A = 1$ if and only if $M \bfi^A = \begin{bmatrix}
0&0&1
\end{bmatrix}^{tr}$.
The set of all such indicator vectors $\bfi^A$ is a coset of the kernel of $M$ and therefore has cardinality $2^{m+n-3}$, as needed.
\end{proof}

\begin{rmk}\label{Rmk:SwitchedLength4}
The facet-defining inequalities for co-faces corresponding to cycles of length 4 in $\dig$
can be obtained from those in Lemma \ref{Lem:Length4Cycles} via the switching operation. We note that every cycle of length 4 in $\dig$ is either of the form described in Lemma \ref{Lem:Length4Cycles} or does not contain the empty set. Therefore, performing a nontrivial switching operation on the inequality in \Cref{Lem:Length4Cycles} always yields a homogeneous inequality, since the empty set is not in the resulting co-face.
\end{rmk}

\begin{cor}\label{Cor:CofaceValues}
Let $\bq \cdot \bx \leq c$ be the facet-defining inequality for a co-facet given by a cycle of length 4 in $\dig$. Then for every $S \subset [m] \sqcup [n]$, either $\bq \cdot \bd^S = c$ or $\bq \cdot \bd^S = c - 2^{m+n-3}$. 
\end{cor}

\begin{proof}
If $c = 2^{m+n-3}$, then the construction of $\Ev(\O, \U)$ shows that $\bq \cdot \bd^S = 0$ whenever $S = \emptyset, \O, \U, \O \sqcup \U$.

Otherwise, $c = 0$ and $\bq$ was obtained by switching another inequality $\bfp \cdot \bx \leq 2^{m+n-3}$ with respect to a set $I$. Let $\bfp \cdot \bx \leq 2^{m+n-3}$ be the facet-defining inequality for the co-facet given by cycle $\{\emptyset, \O, \U, \O \sqcup \U\}$. Then $\bq \cdot \bx \leq 0$ is the facet-defining inequality for the co-facet given by the cycle $\{I, \O \triangle I, \U \triangle I (\O \sqcup \U) \triangle I\}$ where $I \neq \emptyset, \O, \U, \O \sqcup \U$ \cite[Lemma~26.3.3]{cutsDeza1997}. 

Since this inequality is homogeneous, we must have $\bq \cdot \bd^I = -\bfp \cdot \bd^I = -2^{m+n-3}$. 
Let $T \in \{ \O, \U, \O \sqcup \U\}$. For each $S \in \Ev(\O, \U)$, the parity of $I \cap S$ is the same as that of $(T \triangle I) \cap S$. Indeed, for each $S \in \Ev(\O, \U)$, we have that $\#(S \cap \O), \#(S \cap \U)$ and $\#(S \cap (\O \sqcup U))$ are even. The cardinality of $S \cap (T \triangle I)$ is $\#(S \cap T) + \#(S \cap I) - 2\#(S \cap T \cap I)$, which has the same cardinality as $\#(S \cap I)$. So $\bq \cdot \bd^{T \triangle I} = -2^{m+n-3}$ as well.
\end{proof}

\begin{thm}\label{Thm:ADualFacets}
Let $C$ be a directed cycle in $\dig$. Let $C$ be obtained from $C_1$ and $C_2$ by gluing along  edges $A_1 \rightarrow B_1$ in $C_1$ and $A_2 \rightarrow B_2$ in $C_2$  where $A_1, B_1, A_2$ and $B_2$ all have the same parity. Let $F_1$ be the facet of $\Gcut(D_{m,n})$ corresponding to $C_1$ with facet-defining inequality $\bq_1 \cdot \bx \leq c_1$. Let $F_2$ be the facet of $\Gcut(D_{m,n})$ corresponding to $C_2$ with facet-defining inequality $\bq_2 \cdot \bx \leq 0$. Then the facet-defining inequality for the co-face given by $C$ is
\begin{equation}\label{Eqn:CycleFacet}
\bq_1 \cdot \bx + \bq_2 \cdot \bx + \mathbf{a}\cdot \bx \leq c_1,
\end{equation}
where $\mathbf{a}$ is the linear functional given by
\begin{equation}\label{Eqn:ADualFunctional}
\sum_{\substack{S: S \cap A_1 \text{ odd} \\
S \cap B_1 \text{ even} \\
S \cap A_2 \text{ even} \\
S \cap B_2 \text{ odd}}} \bx_S +
\sum_{\substack{S: S \cap A_1 \text{ even} \\
S \cap B_1 \text{ odd} \\
S \cap A_2 \text{ odd} \\
S \cap B_2 \text{ even}}} \bx_S -
\sum_{\substack{S: S \cap A_1 \text{ odd} \\
S \cap B_1 \text{ odd} \\
S \cap A_2 \text{ even} \\
S \cap B_2 \text{ even}}} \bx_S -
\sum_{\substack{S: S \cap A_1 \text{ even} \\
S \cap B_1 \text{ even} \\
S \cap A_2 \text{ odd} \\
S \cap B_2 \text{ odd}}} \bx_S.
\end{equation}
We always have $c_1 = 2^{m+n-3}$ or $c_1 = 0$.
Furthermore, for all $S \subset [m] \sqcup [n]$, the linear functional in Equation (\ref{Eqn:CycleFacet}) evaluated at $\bd^S$ is either $c_1$ or $c_1 - 2^{m+n-3}$.
\end{thm}

Note that Theorem \ref{Thm:ADualFacets} describes every facet of $\Gcut(\Dmn)$. Indeed, every directed cycle in $\dig$ can be obtained by gluing two smaller cycles. At least one of these cycles must not contain the empty set, and so the inequality defining its corresponding facet is homogeneous. Therefore we may always take one of the inequalities to be $\bq_2 \cdot \bx \leq 0$. In order to prove Theorem \ref{Thm:ADualFacets}, we first examine the linear functional $\bfa$ and its values on the vertices of $\Gcut(\Dmn)$.

\begin{lemma}\label{Lem:aEqualsZero}
Let $T \neq A_i \sqcup B_j$ for all $i, j = 1,2$.  Then $\ba \cdot \bd^T = 0$.
\end{lemma}

\begin{proof}
Fix $T \neq A_i \sqcup B_j$ for each $i = 1,2$ and $j = 1,2$. 
We will evaluate $\ba \cdot \bd^T$. 
Let $T = T_1 \sqcup T_2$. 
Consider the matrix $M \in \mathbb{F}_2^{5 \times (m+n)}$ with rows $(\bfi^{A_1})^{tr}, (\bfi^{A_2})^{tr}, (\bfi^{B_1})^{tr}, (\bfi^{B_2})^{tr}$ and $(\bfi^T)^{tr}$. 
Since $A_1 \neq A_2$, $B_1 \neq B_2$, $\supp(A_i) \cap \supp(B_j) = \emptyset$ 
for each $i, j = 1,2$, the matrix $M$ has either rank 4 or 5. 

If $\mathrm{rk}(M) = 5$, then the nonzero terms of each sum in Equation (\ref{Eqn:ADualFunctional}) evaluated at $\bd^S$ correspond to a coset of the kernel of $M$. For example, the nonzero terms of the first sum correspond to those sets $S$ such that $M \bfi^S = \begin{bmatrix}
1 & 0 & 0 & 1 & 1\\
\end{bmatrix}^{tr}.$
Since each of these cosets has the same size, $\ba \cdot \bd^T = 0$.

If $\mathrm{rk}(M) = 4$, then we have $T_1 \in \{ \emptyset, A_1, A_2, A_1 \triangle A_2 \}$ and $T_2 \in \{ \emptyset, B_1, B_2, B_1 \triangle B_2 \}$. There are several cases.

\emph{Case 1:} First, consider the case where $T_1 = \emptyset$.
\begin{enumerate}[label = (\alph*)]
    \item If $T_2 = \emptyset$, then $\ba \cdot \bd^T = 0$, as needed.
    \item If, without loss of generality, $T_2 = B_1$, then $\bfi^{T} \cdot \bfi^{B_1} = 0$. So every term in the first and third sums of Equation (\ref{Eqn:ADualFunctional}) evaluated at $\bd^T$ is 0. The nonzero terms of the second and fourth sums each correspond to cosets of the kernel of $M$. Since these cosets have the same cardinality, $\ba \cdot \bd^T = 0$.
    \item If $T_2 = B_1 \triangle B_2$, then $\bfi^T \cdot \bfi^S = \bfi^{B_1} \cdot \bfi^S + \bfi^{B_2} \cdot \bfi^S = 1$ for all $S$ in the support of $\bfa$. So the nonzero terms of each sum correspond to cosets of the kernel of $M$, and $\ba \cdot \bd^T = 0$.
\end{enumerate}

\emph{Case 2:} Without loss of generality, let $T_1 = A_1$. Then by assumption, we must have $T_2 = \emptyset$ or $T_2 = B_1 \triangle B_2$.
\begin{enumerate}[label = (\alph*)]
    \item If $T_2 = \emptyset$, then this is the same as case 1(b).
    \item If $T_2 = B_1 \triangle B_2$, then $\bfi^T \cdot \bfi^S = \bfi^{A_1} \cdot \bfi^S + \bfi^{B_1} \cdot \bfi^S + \bfi^{B_2} \cdot \bfi^S$ for all $S$. But for all $S$ in the support of $\bfa$, we have $\bfi^{B_1} \cdot \bfi^S + \bfi^{B_2} \cdot \bfi^S = 1.$ So when $\bfi^{A_1} \cdot \bfi^S = 1$ for $S$ in the support of $\bfa$, we cannot have $\bfi^T \cdot \bfi^S = 1$. So every term in the first and third sums of Equation (\ref{Eqn:ADualFunctional}) evaluated at $\bd^T$ are 0. The nonzero terms of the second and fourth sums each correspond to cosets of the kernel of $M$. Since these cosets have the same cardinality, $\ba \cdot \bd^T = 0$.
\end{enumerate}

\emph{Case 3:} Finally, let $T_1 = A_1 \triangle A_2$.
\begin{enumerate}[label = (\alph*)]
    \item If $T_2 = \emptyset$, this is the same as case 1(c).
    \item If $T_2 = B_1$ or $B_2$, this is the same as case 2(b).
    \item If $T_2 = B_1 \triangle B_2$, then $\bfi^T \cdot \bfi^S = \bfi^{A_1} \cdot \bfi^S + \bfi^{A_2} \cdot \bfi^S + \bfi^{B_1} \cdot \bfi^S + \bfi^{B_2} \cdot \bfi^S$. But for all $S$ in the support of $\bfa$, $\bfi^{A_1} \cdot \bfi^S + \bfi^{A_2} = 1$ and $\bfi^{B_1} \cdot \bfi^S + \bfi^{B_2} \cdot \bfi^S = 1$. So $\bfi^T \cdot \bfi^S = 0$ for all $S$ in the support of $\bfa$, and each term in Equation (\ref{Eqn:ADualFunctional}) is 0.
\end{enumerate}
\end{proof}

\begin{prop}\label{Prop:aNonzero}
Let $A_1, A_2 \subset [m]$ and $B_1, B_2 \subset [n]$ such that $A_1 \neq A_2$, $B_1 \neq B_2$ and none of them are empty. Then $\ba \cdot\bd^{A_1 \sqcup B_1} = \ba \cdot \bd^{A_2 \sqcup B_2} = 2^{m+n-3}$ and $\ba \cdot \bd^{A_1 \sqcup B_2} = \ba \cdot \bd^{A_2 \sqcup B_1} = -2^{m+n-3}$.
\end{prop}

\begin{proof}
There are $2^{m+n-3}$ positive terms and $2^{m+n-3}$ negative terms of $\bfa$. The positive terms correspond exactly to those sets $S \subset [m] \sqcup [n]$ such that $\#(S \cap (A_1 \sqcup B_1))$ and $\#(S \cap (A_2 \sqcup B_2))$ are both odd and $\#(S \cap (A_1 \sqcup B_2))$ and $\#(S \cap (A_2 \sqcup B_1))$ are both even. Similarly, the negative terms of $\bfa$ correspond exactly to those sets $S \subset [m] \sqcup [n]$ such that $\#(S \cap (A_1 \sqcup B_2))$ and $\#(S \cap (A_2 \sqcup B_1))$ are both odd and $\#(S \cap (A_1 \sqcup B_1))$ and $\#(S \cap (A_2 \sqcup B_2))$ are both even.
\end{proof}

We can now prove the main result of this section.

\begin{proof}[Proof of Theorem \ref{Thm:ADualFacets}]
Let $4k$ be the length of the cycle $C$. If $k = 1$, then by Lemma \ref{Lem:Length4Cycles} and the remark that follows it, the facet-defining inequality for the facet given by $C$ is either of the form $\bq \cdot \bx \leq 2^{m+n-3}$ or $\bq \cdot \bx \leq 0$. By Corollary \ref{Cor:CofaceValues}, for each $S \subset [m] \sqcup [n]$, we either have $\bq \cdot \bd^S = c_1$ or $\bq \cdot \bd^S = c_1 - 2^{m+n-3}$.
Now we proceed by induction on $k$.

Let $C$ be obtained by gluing $C_1$ and $C_2$ along $A_1 \sqcup B_1$ and $A_2 \sqcup B_2$. Let $\bq_1 \cdot \bx \leq c_1$ be the facet-defining inequality for co-face $C_1$ and let $\bq_2 \cdot \bx \leq 0$ be the facet-defining inequality for co-face $C_2$. Let $\ba$ be defined as in the statement of the theorem. Let $\Bell = \bq_1 + \bq_2 + \ba$.

First, let $S$ be a set that does not label an edge of $C$. We wish to show that $\Bell \cdot \bd^S = c_1$.
Either  $S = A_1 \sqcup B_1$, $S = A_2 \sqcup B_2$, or $S$ not an edge of $C_1$ or $C_2$.

If $S$ is neither an edge of $C_1$ nor $C_2$, then $\bd^S$ is on both of the facets corresponding to these co-facets. Therefore we have
$\bq_1 \cdot \bd^S = c_1$ and $\bq_2 \cdot \bd_2 = 0$. By Lemma \ref{Lem:aEqualsZero}, $\ba \cdot \bd^S = 0$, which concludes that  $\Bell \cdot \bd^S = c_1$.

If $S = A_1 \sqcup B_1$, then $S$ is an edge of $C_1$. By induction, this implies that $\bq_1 \cdot \bd^S = c_1 - 2^{m+n-3}$. Since $S$ is not an edge of $C_2$, $\bq_2 \cdot \bd^S = 0$. By Proposition \ref{Prop:aNonzero}, $\ba \cdot \bd^S = 2^{m+n-3}$. So $\Bell \cdot \bd^S = c_1$.

Finally, if $S = A_2 \sqcup B_2$, then $S$ is not an edge of $C_1$, so $\bq_1 \cdot \bd^S = c_1$. Since $S$ is an edge of $C_2$, we have  by induction that $\bq_2 \cdot \bd^S = - 2^{m+n-3}$. By Proposition \ref{Prop:aNonzero}, $\ba \cdot \bd^S = 2^{m+n-3}$. So $\Bell \cdot \bd^S = c_1$.

Now we wish to show that if $S$ \emph{is} an edge of $C$, then $\Bell \cdot \bd^S = c_1 - 2^{m+n-3}$. In particular, this shows that $\Bell \cdot \bx \leq c_1$ is a valid inequality over the polytope and that all vertices of $\Gcut(\Dmn)$ either satisfy $\Bell \cdot \bd^S = c_1$ or $\Bell \cdot \bd^S = c_1 - 2^{m+n-3}$.

If $S$ is an edge of $C_1$, then it is not an edge of $C_2$. So we have $\bq_1 \cdot \bd^S = c_1 - 2^{m+n-3}$ by induction, and $\bq_2 \cdot \bd^S = 0$. By Lemma \ref{Lem:aEqualsZero}, $\ba \cdot \bd^S = 0$. So $\Bell \cdot \bd^S = c_1 - 2^{m+n-3}$.

If $S$ is an edge of $C_2$, then it is not an edge of $C_1$. So we have $\bq_1 \cdot \bd^S = c_1$. By induction, we have $\bq_2 \cdot \bd^S = -2^{m+n-3}$. By Lemma \ref{Lem:aEqualsZero}, $\ba \cdot \bd^S = 0$. So $\Bell \cdot \bd^S = c_1 - 2^{m+n-3}$.

Finally if $S = A_1 \sqcup B_2$ or $S = A_2 \sqcup B_1$, then $S$ is neither an edge of $C_1$ nor $C_2$. So $\bq_1 \cdot \bd^S = c_1$ and $\bq_2 \cdot \bd^S = 0$. By Proposition \ref{Prop:aNonzero}, $\ba \cdot \bd^S = -2^{m+n-3}$. So $\Bell \cdot \bd^S = c_1 - 2^{m+n-3}$, as needed.
\end{proof}

\begin{ex}
Let $m = n = 2$ and consider the length 4 cycles in $G^{2,2}_2$ pictured in Figure \ref{Fig:CycleGluing}. These cycles are $C_1 = \{\emptyset, 1, 3, 13\}$ and $C_2 = \{24, 124, 234, 1234\}.$ The cycle obtained by gluing $C_1$ and $C_2$ along $13$ and $24$ is $\{ \emptyset, 1, 3, 14, 23, 124, 234, 1234\}$. By Lemma \ref{Lem:Length4Cycles}, the facet-defining inequality for co-face $C_1$ is $\bq_1 \cdot \bx = x_2 + x_4 + x_{24} \leq 2$.
The co-face $C_2$ is obtained by switching $C_1$ with respect to $1234$. Indeed, taking the symmetric difference of each element of $C_1$ with respect to $1234$ yields $C_2$. The switching of the facet defining inequality is $\bq_2 \cdot \bx -x_2 - x_4 + x_{24} \leq 0.$

The linear functional $\bfa$ as described in Theorem \ref{Thm:ADualFacets} is $\bfa = -x_{13} + x_{14} + x_{23} + x_{24}.$
So the facet-defining inequality for the co-face $C$ is
\[
\Bell \cdot \bx = (\bq_1 + \bq_2 + \bfa) \cdot \bx = -x_{13} + x_{14} + x_{23} + x_{24} \leq 2.
\]
We can write $\Bell$ as the row vector $[ 0, 0, 0, 0, -1, 1, 1, 1 ]$.
The map that sends $\Corr(\Dmn)$ to $\Gcut(\Dmn)$ is
\[
\Phi = 
\begin{blockarray}{ccccccccc}
& 1 & 2 & 3 & 4 & 13 & 14 & 23 & 24 \\
\begin{block}{c(cccccccc)}
1 & 1 & 0 & 0 & 0 & 0 & 0 & 0 & 0 \\
2 & 0 & 1 & 0 & 0 & 0 & 0 & 0 & 0 \\
3 & 0 & 0 & 1 & 0 & 0 & 0 & 0 & 0 \\
4 & 0 & 0 & 0 & 1 & 0 & 0 & 0 & 0 \\
13 & 1 & 0 & 1 & 0 & -2 & 0 & 0 & 0 \\
14 & 1 & 0 & 0 & 1 & 0 & -2 & 0 & 0 \\
23 & 0 & 1 & 1 & 0 & 0 & 0 & -2 & 0 \\
24 & 0 & 1 & 0 & 1 & 0 & 0 & 0 & -2 \\
\end{block}
\end{blockarray}.
\]

So the inequality defining co-facet $C$ in $\Corr(\Delta)$ is
\[
(\Bell \Phi) \bfy = 2y_2 + 2y_4 + 2 y_{13} - 2y_{14} - 2y_{23} - 2y_{24} \leq 2,
\]
or $y_2 + y_4 + y_{13} - y_{14} - y_{23} - y_{24} \leq 1$. So we have the facet-defining linear functional $\bfp = [0,1,0,1,1,-1,-1,-1]$.

A pseudoinverse for the matrix sending the translated marginal polytope, $\Marg^0(\Dmn)$ to $\Corr(\Dmn)$ as defined in the appendex in \Cref{Prop:MargToCorr} is 
\[
\Omega = \begin{blockarray}{ccccccccccccccccc}
\begin{block}{c \BAmulticolumn{4}{>{\bf}c}\BAmulticolumn{4}{>{\bf}c}\BAmulticolumn{4}{>{\bf}c}\BAmulticolumn{4}{>{\bf}c}}
& 13 & 14 & 23 & 24 \\
& \overbrace{\qquad \qquad \qquad \qquad} & \overbrace{\qquad \qquad \qquad \qquad } & \overbrace{\qquad \qquad \qquad \qquad} & \overbrace{\qquad \qquad \qquad \qquad} \\
\end{block}
& \emptyset & 1 & 3 & 13 & \emptyset & 1 & 4 & 14 & \emptyset & 2 & 3 & 23 & \emptyset & 2 & 4 & 24 \\
\BAhline
\begin{block}{c(cccccccccccccccc)}
1&0&1/2&0&1/2&0&1/2&0&1/2&0&0&0&0&0&0&0&0\\
2&0&0&0&0&0&0&0&0&0&1/2&0&1/2&0&1/2&0&1/2\\
3&0&0&1/2&1/2&0&0&0&0&0&0&1/2&1/2&0&0&0&0\\
4&0&0&0&0&0&0&1/2&1/2&0&0&0&0&0&0&1/2&1/2\\
13&0&0&0&1&0&0&0&0&0&0&0&0&0&0&0&0\\
14&0&0&0&0&0&0&0&1&0&0&0&0&0&0&0&0\\
23&0&0&0&0&0&0&0&0&0&0&0&1&0&0&0&0\\
24&0&0&0&0&0&0&0&0&0&0&0&0&0&0&0&1\\
\end{block} 
\end{blockarray}.
\]

The inequality defining co-facet $C$ in $\Marg(\Delta)$ is
\begin{align*}
\bfr \cdot \bfz & := (\bfp \Omega) \bfz \\
&= z_{(13,13)} + \frac{1}{2}z_{(4,14)} - \frac{1}{2}z_{(14,14)} + \frac{1}{2}z_{(2,23)} - \frac{1}{2}z_{(23,23)} + \frac{1}{2}z_{(2,24)} + \frac{1}{2}z_{(4,24)} \leq 1.
\end{align*}

\end{ex}

The Lawrence lifting operation also preserves unimodularity of a simplicial complex \cite{unimodHierModels2017}. For this reason, we would like to understand how the $\cH$-representation of $\Gcut(\Delta)$ relates to that of $\Gcut(\Lambda(\Delta))$. 

\begin{defn}
The Lawrence lifting of a simplicial complex $\D$ on ground set $[n]$, denoted $\Lambda(\Delta)$, is the simplicial complex on the set $[n]\cup \{\ell\}$ with facets
$\{[n] \} \cup\{F\cup \{\ell\}~|~ F\in \facet(\D)\}.$
In particular, the Lawrence lifting of $2^{[m]}\sqcup 2^{[n]}$ is a simplical complex with facets 
\begin{align*}
  \facet(\Lambda(2^{[m]} \sqcup 2^{[n]})) =  \{[m]\sqcup [n], [m]\cup\{\ell\},[n]\cup \{\ell\} \}.
\end{align*}
For example, Figure \ref{fig:LawAlex} shows the Lawrence lifting of the disjoint union of a $0$-simplex with a $1$-simplex. 
\end{defn}

A unified $\cH$-representation for $\Gcut(\Lambda(2^{[m]} \sqcup 2^{[n]}))$, combined with the $\cH$-representation of $\Gcut(\Dmn)$ given in the present paper, could prove useful for understanding the facet descriptions of binary marginal polytopes of arbitrary unimodular simplicial complexes. 
 In order to obtain a co-facet description for the generalized cut polytope, one can study the oriented matroid for $\Gcut(\Delta)$ since any signed circuit in $\mathcal{B}$ corresponds to two positive circuits in $\mathcal{B}(\Lambda(\D))$. One also obtains the two-element circuits $\{\bfb_S,\bfb_{S\sqcup \{p\}}\}$ for any $S\subset [n]$. In the case where $\Delta = 2^{[m]} \sqcup 2^{[n]}$, the matroid underlying the Gale dual is well-understood \cite{unimodMatriod2017}.
 Another approach may be to find $\cH$-representations of the no-three-way interaction model (see \Cref{prop:3cycledegree}) and adapt the results for our case. 
 
 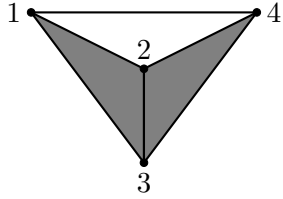
\begin{figure} 
    \centering
\begin{tikzpicture}
    \draw[thick, fill=gray] (1,0)--(-.5,2)--(1,1.25)--(1,0);
    \draw[thick, fill=gray]
    (1,0)--(2.5,2)--(1,1.25)--(1,0);
    \draw[thick] (2.5,2)--(-.5,2);
    \draw[fill] (1,0) circle [radius = .05];
    \node[below] at (1,0) {3};
    \draw[fill] (-.5,2) circle [radius = .05];
    \node[left] at (-.5,2) {1};
    \draw[fill] (2.5,2) circle [radius = .05];
    \node[right] at (2.5,2) {4};
    \draw[fill] (1,1.25) circle [radius = .05];
    \node[above] at (1,1.25) {2};
\end{tikzpicture}
 \caption{The Lawrence lifting of $2^{[1]}\sqcup 2^{[2]}$.}\label{fig:LawAlex}
\end{figure}

\section{The Degree of Some Binary Hierarchical Models}
 
This section discusses the degrees of some binary hierarchical models. This is the algebraic degree  of the toric ideal with the design matrix of the model as its presentation matrix (see \cite[Chapter~4]{sturmfels2008}). Among other uses, the degree of the hierarchical model serves as an upper bound for its maximum likelihood degree \cite[Corollary~8]{toricMLDegree}. Develin and Sullivant in \cite{develin2003}  compute the degree of the binary graph models when the underlying graph is a forest. In \Cref{thm:degreeCone} we connect the degree of a hierarchical model and its cone.  In \Cref{cor:degreeTurtle} we compute the degree of a binary hierarchical model when the underlying structure is a turtle complex. 
The rest of the section gives examples and discusses  the degree of unimodular hierarchical models that are Lawrence liftings and Alexander duals of the disjoint union of two simplices.

Let  $U_{\D}$ be  the design matrix for the binary hierarchical model with simplicial complex $\D$ on ground set $[n]$. Take the polynomial rings $R_n=\K[p_{S}~|~ S\subset [n]]$ and $\K[\theta_{(H,F)}~|~ H\subset F\in \facet(\D)]$ over some field $\K$, and consider the toric map between them: 
  \begin{equation}
  \label{eq:toric}
 \begin{split}
 f_{\D} \colon R_{n}=\K[p_{S} \mid  S\subset [n]] & \xrightarrow{~U_{\D}~}
  \K[\theta_{(H,F)}~|~ H\subset F\in \facet(\D)],  \\
 p_{S}& \longmapsto \prod_{F\in\facet(\D)}\theta_{(S\cap F,F)}.
 \end{split}
 \end{equation}
The \emph{toric ideal} $I(\D)$ for the hierarchical model $\cM(\D)$ and the marginal polytope $\Marg(\D)$ is the kernel of the map $f_{\D}$. The dimension of this ideal, denoted $\dim_{\K}(R_n\slash I(\D))$, is equal to the number of faces in $\D$ since it is one more than the dimension of the polytope $\Marg(\D)$ discussed in \Cref{thm:dim}. The \emph{degree} of $I(\D)$, denoted $\deg(I(\D))$,  is the number of points in the intersection of its defining variety with $\dim_{\K}(R_n\slash I(\D))$  hyperplanes in general position, counted with multiplicity. This is also the normalized volume of $\Marg(\D)$ \cite{sturmfels1996}. Algebraically, the dimension and the degree of $I(\D)$ are recorded in the reduced form of the Hilbert series for  $I(\D)$. Recall that the Hilbert series for $I(\D)$ is the formal power series 
\[ 
H_{R_n\slash I(\D)}(t)=\sum\limits_{d \ge 0}\dim_{\K} [{R_n\slash I(\D)}]_d t^d,
\]
where $[{R_n\slash I(\D)}]_d $ is the $d$-th graded component of $R_n\slash I(\D)$. 
By Hilbert's theorem \cite[Corollary 4.1.8]{conca1994}, this series is rational and can be uniquely written in its reducible form as 
\begin{align}
\label{eq:hilbert}
H_{R_n\slash I(\D)}(t)= \dfrac{k(t)}{(1-t)^{\dim_{\K}(R_n\slash I(\D))}},    
\end{align}
with $k(t)\in \mathbb{Z}[t]$ and $k(1) > 0$, unless $I(\D) = R_n$. The evaluation  of $k(t)$ at $t=1$ is the degree of $I(\D)$.

\begin{rmk}
One can similarly associate toric ideals to the correlation polytope and the generalized cut polytope.  \Cref{thm:dim} implies that these ideals share toric geometric properties, including their dimension and degree.
\end{rmk}

We start off by computing  in \Cref{ex:degreeDisjointUnion} the degree of the ideal for the disjoint union of two simplices. As we will see, this ideal is isomorphic to the Segre embedding of two projective planes, whose degree is computed by Herzog and Trung in \cite{herzog1992}.  The example gives an alternative shorter proof of this result using the rational form of the Hilbert series for the Segre embedding found in \cite{conca1994}. \Cref{ex:alexdegree} computes the degree when the simplicial complex is the Alexander dual $D_{1,n}$. 

\begin{ex}
\label{ex:degreeDisjointUnion}
The ideal  $I({2^{[m]}\sqcup2^{[n]}})$ has degree $\binom{2^{m}+2^{n}-2}{2^{m}-1}$. 
\end{ex}

\begin{proof}
The toric map $f_{\D}$  in \Cref{eq:toric} for the binary hierarchical model $\Marg(2^{[m]}\sqcup2^{[n]})$ has form:
  \begin{equation*} 
 \begin{split}
 f_{\D} \colon R_{m+n}=\K[p_{S\cup T} \mid  S\subset [m], T\subset [n]] & \xrightarrow{~U_{\D}~}
 \K[\theta_{(S,[m])},\theta_{(T,[n])} ~|~ S\subset [m], T\subset [n]],  \\
 p_{S\cup T}& \longmapsto \theta_{(S,[m])}\cdot \theta_{(T,[n])}.
 \end{split}
 \end{equation*}
Up to renaming of the variables, the above is the Segre embedding of \(\mathbb{P}^{2^m-1}\times \mathbb{P}^{2^n-1}\) in  \cite[Proposition~9.45]{hassett2007}.  
Its reduced Hilbert series has form (see \cite{conca1994}):
\begin{align*}
   H_{R_{m+n}\slash I(\D)}(t)= \dfrac{\sum\limits_{i=0}^{2^m-1}{\binom{2^{m}-1}{i}}{\binom{2^{n}-1}{i}}t^i}{(1-t)^{2^{m}+2^{n}-1}},
\end{align*}
which implies that the degree of the ideal is:
\begin{align*}
   \deg(I(\D))&= \sum\limits_{i=0}^{2^m-1}{{2^{m}-1}\choose{i}}{{2^{n}-1}\choose{i}}
   &= \sum\limits_{i=0}^{2^m-1}{{2^{m}-1}\choose{2^m-1-i}}{{2^{n}-1}\choose{i}}
   &= {{2^{m}+2^{n}-2}\choose{2^{m}-1}},
\end{align*}{}
The last equality is the Vandermonde's identity.
\end{proof}

\begin{ex}
\label{ex:alexdegree}
The ideal $I(D_{1,n})$ has degree $2^{n-1}$.
\end{ex}

\begin{proof}
According to \Cref{sec:alexDual} $D_{1,n}=\{ S \sqcup T \mid S \subsetneq [1], T \subsetneq [n]\}$ is the boundary of the simplex $2^{[n]}$. By  \cite[Theorem~2.8]{reducCyclic2002}, the ideal $I(2^{[n]}\setminus [n])$ is generated by only one binomial of degree  $2^{n-1}$. Hence,  the degree of our ideal is $2^{n-1}$. 
\end{proof}
The following theorem computes the degree of $I({\cC^{k}(\D))})$  using the degree of  $I(\D)$.
\begin{thm}
\label{thm:degreeCone}
 The degree of $I({\cC(\D))}$  is  $\deg(I(\D))^2$. 
\end{thm}

\begin{proof}
Let $f_{\cC(\D)}$ be the toric map for $\cC(\D)$. 
 \begin{equation*}
 \begin{split}
 f_{\cC(\D)} \colon R_{n+1}=\K[p_{S},p_{S\cup \{\ell\}} \mid  S\subset [n]] & \xrightarrow{~U_{\cC(\D)}~}
 \K[\theta_{(H,F\cup\{\ell\})},\theta_{(H\cup\{\ell\},F\cup\{\ell\})} ~|~ H\subset F\in \facet(\D)],  \\
 p_{S}& \longmapsto \prod\limits_{F\in\facet(\D)}\theta_{(S\cap F,F\cup \{\ell\})},\\ p_{S\cup\{\ell\}}& \longmapsto \prod\limits_{F\in\facet(\D)}\theta_{(S\cap F\cup \{\ell\},F\cup \{\ell\})}.
 \end{split}
 \end{equation*}

The set of variables $\{p_{S}~|~ S\subset [n]\}$ gets mapped to monomials in variables $\{\theta_{(H,F\cup\{\ell\})}~|~H\subset F\in \facet(\D) \}$. The set of variables $\{p_{\cup\{\ell\}}~|~ S\subset [n]\}$ gets mapped to monomials in variables $\{\theta_{(H\cup\{\ell\},F\cup\{\ell\})}~|~H\subset F\in \facet(\D) \}$. Since these sets of variables are disjoint, we may consider the two restricted maps

 \begin{equation*} 
 \begin{split}
 f_{\D} \colon R_{n}=\K[p_{S} \mid  S\subset [n]] & \xrightarrow{~U_{\D}~}
 \K[\theta_{(H,F\cup\{\ell\})} ~|~ H\subset F\in \facet(\D)],  \\
 p_{S}& \longmapsto \prod\limits_{F\in\facet(\D)}\theta_{(S\cap F,F\cup \{\ell\})},
 \end{split}
 \end{equation*}
 and 
  \begin{equation*} 
 \begin{split}
 f_{\D\sqcup \{\ell\}} \colon R'_{n}=\K[p_{S\cup\{\ell\}} \mid  S\subset [n]] & \xrightarrow{~U_{\D}~}
 \K[\theta_{(H\cup\{\ell\},F\cup\{\ell\})} ~|~ H\subset F\in \facet(\D)],  \\
 p_{S\cup\{\ell\}}& \longmapsto \prod\limits_{F\in\facet(\D)}\theta_{(S\cap F\cup \{\ell\},F\cup \{\ell\})}.
 \end{split}
 \end{equation*}
The kernel of the first map is $I(\D)$. Denote the kernel of the second map with $I(\D\sqcup \ell)$.  The isomorphism from $R_n$ to $R'_n$ that maps each $p_S$ to $p_{S\cup\{\ell\}}$  maps $I(\D)$ to $I_{\D\sqcup \ell}$. 
$R_n$ and $R'_n$ (consequently  $I(\D)$ and $I({\D\sqcup \{\ell\}})$) have zero intersection in $R_{n+1}$. Standard results from commutative  algebra give 
\begin{align*}
  R_{n+1}\slash I(\cC(\D)) &= R_{n+1}\slash (I(\D)+I(\D\sqcup p))\\
  &\cong (R_{n}\slash I(\D)) \otimes_{R_{n+1}} (R'_{n}\slash I(\D\sqcup p))\\
  &\cong (R_{n}\slash I(\D)) \otimes_{R_{n+1}} (R_{n}\slash I(\D)).
\end{align*}

From here, given the Hilbert series of $I(\D)$ in the form \Cref{eq:hilbert}, a  rational  form of the Hilbert series for $I(\cC(\D))$ is

\[H_{R_{n+1}\slash I_{\cC(\D)}}(t)=[H_{R_n\slash I_{\D}}(t)]^2=\dfrac{k(t)^2}{(1-t)^{2dim(R_n\slash I_{\D})}}.\]
This is the  reduced Hilbert series for $I(\cC(\D))$ since  power of the  factor $(1-t)$ in the denominator is equal to the dimension of $I({\cC(\D)})$. Hence, the degree of $I(\cC(\D))$ is $(k(1))=(\deg(I(\D))^2$.
\end{proof}

An immediate consequence of \Cref{thm:degreeCone} is the degree of ideals for turtle complexes. 

\begin{cor}
\label{cor:degreeTurtle}
The ideal $I_{\mathbb{T}_n^k}$ has degree $2^{(k-1)2^{(n-k)}}$.
\end{cor}

\begin{proof}
From \Cref{obs:turtleCones}, we need to compute the degree of $I({\D})$ for $\D=\cC^{n-k}(2^{[k]}\setminus [k])$. The ideal $I({2^{[k]}\setminus [k]})$ has degree $2^{k-1}$ since it is generated by one binomial of degree $2^{k-1}$ \cite[Theorem~2.8]{reducCyclic2002}.  Applying \Cref{thm:degreeCone} $n-k$ times inductively over $2^{[k]}\setminus [k]$ in the same order indicated by \Cref{eq:cone} concludes that the degree of $I(\cC^{n-k}(2^{[k]}\setminus [k]))$ is $2^{(k-1)2^{(n-k)}}$. 

\end{proof}

In the last part of this section we discuss the degree of the ideal for the Lawrence lifting of the disjoint union of two simplices. Intuitively, by contracting the two simplices to points, one transforms the Lawrence lifting of the disjoint union of two simplices  to the no-three-way interaction hierarchical models, and obtains the following result.  

\begin{lemma}
The marginal polytopes $\Marg(\Lambda(2^{[m]}\sqcup2^{[n]}))$ and $\Marg([12][23][13],(2,2^m,2^n))$ are isomorphic.  
\end{lemma}

\begin{proof}
Recall the string definition of the marginal polytope in \Cref{eqn:margEntries}. Let $f$ be a bijection from $\{0,1\}^{m}$ to $[2^m]$ and let $g$ be a bijection from $\{0,1\}^{n}$ to $[2^n]$.  The induced map $(\mathrm{id}_{[2]},f,g)$ applied to the indices of vectors $\mathbf{u}^{i,\mathbf{j},\mathbf{k}}$, for $(i,\mathbf{j},\mathbf{k})\in \{0,1\}\times \{0,1\}^{m} \times  \{0,1\}^n$  serves as a bijection between  $\Marg(\Lambda(2^{[m]}\sqcup2^{[n]}))$ and $\Marg([12][23][13],(2,2^m,2^n))$
\end{proof}

The no-three-way interaction models are famous for being challenging models in algebraic statistics, and very little is known about their  degree. The following is the most current result on this topic.  
\begin{prop}\cite[Proposition~34]{toricMLDegree} 
\label{prop:3cycledegree}
The degree of no-three-way interaction hierarchical  model $\cM(\{[12][13][23]\},(2,2,n))$ is $n2^{n-1}$. 
\end{prop}

The equivalence among $\Lambda(2^{[m]}\sqcup 2^{[n]})$ and the  no-three-way-interaction models indicates the level of difficulty for our problem, and it offers a new perspective. In particular, one immediate  consequence of \Cref{prop:3cycledegree} is that the degree of $I(\Lambda(2^{[1]}\sqcup 2^{[n]})$ is $2^{2^{n}+n-1}$. Computations in Macaulay2 \cite{Macaulay2} show that the degree of  $I(\Lambda(2^{[2]\sqcup}2^{[2]}))$ is $4096$. We end this paper by conjecturing that in general the degree of  $I(\Lambda(2^{[m]}\sqcup 2^{[n]}))$ is $2^{m(2^n-1)+n(2^m-1)}$. It is computationally challenging to check the conjecture for larger values of $m$ and $n$, since the degree is expected to be very large.

\section*{Acknowledgements}

The authors would like to thank Daniel Irving Bernstein, Benjamin Braun, Christopher Manon, and Seth Sullivant for many helpful conversations. We are also grateful to the Triangle Lectures in Combinatorics, which brought our group together and facilitated our work on this project. Jane Ivy Coons was partially supported by the US National Science Foundation (DGE 1746939). Benjamin Hollering was partially supported by the US National Science Foundation (DMS 1615660). Aida Maraj was partially supported by the Max-Planck-Institute for Mathematics in the Sciences.

\appendix

\section{Linear Transformations among Correlation, Marginal, and the Generalized Cut Polytope}

In this appendix, we describe the maps among the binary marginal polytope, correlation polytope, and generalized cut polytope associated to a simplicial complex $\D$.

Take the linear transformation from $\R^{\Bar{\D}}$ to itself defined by the matrix $\Psi_{\Delta}$ with entries
\[(\Psi_{\Delta})^{G}_{H}= \begin{cases}{}
\dfrac{(-1)^{\#G -1}}{2^{\#H-1}} & if \: G\subset H\\
0 & if \: G\nsubseteq H.
\end{cases}
\]

\begin{prop}\label{Prop:CutToCorr}
The linear map $\Psi_{\Delta}$ sends $\Gcut(\Delta)$ to $\Corr(\Delta)$.
\end{prop}

\begin{proof}
Let $\Psi = \Psi_\Delta$. We will  prove that $\Psi$ is the inverse of the matrix $\Phi:=\Phi_{\D}$ for the generalized covariance map in \Cref{def:gencovmap}, by showing that their product is the identity matrix. Take  $F,G\in \Bar{\D}$ and compute the entries of this product:
\begin{align*}
    (\Phi\cdot\Psi)_F^G=\Phi_F\cdot \Psi^G=\sum_{H\in \Bar{\D}} \Phi_F^H\Psi^G_H.
\end{align*}{}
 
The entries $\Phi_F^H$ and $\Psi_H^G$  are both nonzero only  if $G\subset H\subset F$. When $F=G$, i.e. in the diagonal entries of $\Phi\cdot\Psi$, the 
product $\Phi_F^H\Psi^F_H$ is nonzero only when $H=F$, which produces  $ (\Phi\cdot \Psi)_F^F=1$. When $F\neq G$  one has 
\begin{align*}
    (\Phi\cdot\Psi)_F^G&= \sum_{G\subset H\subset F}(-2)^{\#H}\cdot \dfrac{(-1)^{\#G-1}}{2^{\#H-1}}\\ =& (-1)^{\#G-1}\sum_{G\subset H\subset F}(-1)^{\#H-1}\\
    =&(-1)^{\#G-1}\sum_{G\subset H\subset F }(-1)^{\#G+ \#(H\backslash  G)-1}\\
    = &(-1)^{2(\#G-1)}\sum_{H' \subset F\backslash  G }(-1)^{\#G-1}(-1)^{\#(H')} \\
    =&1\cdot 0=0. 
\end{align*}{}
\end{proof}

For each $T \in \Delta$, let $f(T)$ denote the number of facets $F \in \Delta$ such that $T \subset F$. 
Define the matrix $\Omega_{\Delta} \in \R^{\overline{\D} \times \mathcal{E}(\D)}$ by
\[
(\Omega_\Delta)^{(H,F)}_T = \begin{cases}
1/f(T) & \text{ if } T \subset H \\
0 & \text{ otherwise.}
\end{cases}
\]

\begin{prop}\label{Prop:MargToCorr}
The linear map $\Omega_\Delta$ sends $\Marg(\Delta)$ to $\Corr(\Delta)$.
\end{prop}

\begin{proof}
Let $\Omega = \Omega_\Delta.$
Fix $T \in (\bar\Delta)$ and $S \subset [n]$. 
Consider the product of the row $\Omega_T$ with $\bu^S$.
We must show that
\[
\Omega_T \bu^S = \begin{cases}
1 & \text{ if } T \subset S \\
0 & \text{ otherwise.}
\end{cases}
\]
For each $(H,F) \in \mathcal{E}(\Delta)$, we have that 
\[
\Omega_T^{(H,F)} \bu^S_{(H,F)} = \begin{cases}
1/f(T) &\text{ if } T \subset H \text{ and } S \cap F = H \\
0 & \text{ otherwise.}
\end{cases}
\]
First suppose that $T$ is not a subset of  $S$. Then for any facet $F$, we cannot have $T \subset S \cap F$. Hence, the scalar  $\Omega_T^{(H,F)} \bu^S_{(H,F)}$ must be zero for all $H$ and $F$, which implies that $M_T \bu^S $ is zero. 

Now suppose that $T$ is a subset of $S$. For each $F\in \facet(\D)$ such that $T \subset F$, there is exactly one $H$ such that $\Omega_T^{(H,F)} \bu^S_{(H,F)} \neq 0$, namely $H = S \cap F$. If $T$ is not a subset of $F$, then $\Omega_T^{(H,F)} \bu^S_{(H,F)} = 0$ for all $H$. So
\[
\Omega_T \bu^S = \sum_{\substack{F \text{ facet} \\ T \subset F}} 1/f(T) = 1.
\]
Therefore the image of $\Marg(\Delta)$ under $\Omega$ is $\Corr(\Delta)$.
\end{proof}

Since $\Marg(\Delta)$ does not contain the origin, the map from $\Corr(\D)$ to $\Marg(\D)$ must be an affine map.
Consider the affine transformation $\Pi_\Delta \bx+\bu^{\emptyset}$ from $\R^{\bar\Delta}$ to $\R^{\mathcal{E}(\D)}$ with  
\[({\Pi_\Delta})^{T}_{(H,F)}=\begin{cases} 
      (-1)^{\#H+\#T} & if\: H\subset T\subset F\\
      0 & \mathrm{else},
   \end{cases}
\]
\begin{prop}
\label{Prop:CorrToMarg}
The affine map $\Pi_\Delta \x+\bu^{\emptyset}$ sends $\Corr(\Delta)$ to $\Marg(\Delta)$ .
\end{prop}

\begin{proof}
Let $\Pi = \Pi_\Delta$. Consider the polytope $\Marg^{\emptyset}(\D)$ that arises from translating $\Marg(\D)$  by $\bu^{\emptyset}$.   Its vertices $\bfw^S=\bfu^S-\bfu^{\emptyset}$, for any $S\subset [n]$, have coordinates in $\R^{E}$ 
\[w^{S}_{(H,F)}=\begin{cases} 
      1 & if\: S\cap F=H, H\neq \{\emptyset\}\\
      -1 & if \: S\cap F\neq H, H=\{\emptyset\},\\
      0 & \mathrm{else}.
   \end{cases}
\]

To prove that the proposed affine transformation maps $\Corr(\D)$ to $\Marg(\D)$, it is enough to prove that the linear transformation $\Pi$  maps $\Corr(\D)$ to $\Marg^{\emptyset}(\D)$. Since $\Pi^{T}_{(H,F)}$ and $\bv^T_{S}$ are nonzero only when $H\subset T\subset F$, and $T\subset S$ we have the following situation: 
\begin{align}
\label{eq:cortomarg}
    \Pi_{(H,F)}\cdot \bv^{S}=\sum_{T\in \bar{\D}}\Pi^{T}_{(H,F)} v^S_{T}=\sum_{\substack{T\in\bar{\D}\\ H\subset T\subset F\cap S}}\Pi^{T}_{(H,F)} v^S_{T}.
\end{align}
From here, we have four possible scenarios: 

\textit{Case 1:} Let $S\cap F = H$ and $H= \emptyset$.  In this case \(\Pi_{(\emptyset,F)}\cdot \bv^{S}=0\) since there is no  nonempty face  $T \in \D$ with $ H\subset T\subset F\cap S$.

\textit{Case 2:}  Let $S\cap F=H$ and $ H\neq \{\emptyset\}$. In this case the only the nonempty face  $T \in \D$  that satisfies $ H\subset T\subset F\cap S$ is $H$. Hence,  
\(\Pi_{(H,F)}\cdot \bv^{S}=(-1)^{\#H+\#H}=1\).

\textit{Case 3:} Let $S\cap F \neq H$ and $ H= \{\emptyset\}$. Then
\[\Pi_{(\emptyset,F)}\cdot \bv^{S}=\sum_{\substack{T\in\Bar{\D}\\ T\subset F\cap S}}(-1)^{\#T} =\sum_{\emptyset\neq T\subset F\cap S }(-1)^{\#T} =\sum_{ T\subset F\cap S }(-1)^{\#T}-1=-1.
\]

\textit{Case 4:} $S\cap F \neq H, \: H\neq  \emptyset$.  Then
\begin{align*}
\Pi_{(H,F)}\cdot \bv^{S}=\sum_{\substack{T\in\Bar{\D}\\ H\subset T\subset F\cap S}}(-1)^{\#T+\#H} =(-1)^{\#H}\sum_{\substack{T\in\Bar{\D}\\ H\subset T\subset F\cap S}}(-1)^{\#T} = (-1)^{2\#H}\sum_{T'\subset (F\cap S)\backslash  H}(-1)^{\#T'}
=0.   
\end{align*}{}
So we have shown that $\Pi \cdot \bv^S = \bw^S,$ as needed.
\end{proof}

Finally, the compositions of the transformations above produce the maps between the generalized cut polytope and  the binary marginal polytope of $\D$. 

\nocite{*}
\bibliographystyle{acm}
\bibliography{paper}
\end{document}